\documentclass[a4paper,11pt]{article}
\usepackage{amsmath}                    
\usepackage{amssymb}                    
\usepackage{amsthm}                     
\usepackage{amsfonts}                   
\usepackage{subcaption}
\usepackage{color}
\usepackage{graphics,graphicx}
\usepackage{enumerate}
\usepackage{url}
\usepackage[unicode,a4paper,pdfstartview = FitH]{hyperref}

\allowdisplaybreaks

\newtheorem{theorem}{Theorem}[section]
\newtheorem{definition}[theorem]{Definition}
\newtheorem{pro}[theorem]{Proposition}
\newtheorem{lemma}[theorem]{Lemma}
\newtheorem{corollary}[theorem]{Corollary}
\newtheorem{remark}[theorem]{Remark}
\newtheorem{ass}[theorem]{Assumption}

\def\ts{\thinspace}

\newcommand{\R}{\mathbb{R}}
\newcommand{\bigo}{\mathcal{O}}

\def \IR{\mathbb R}

\def \IC{\mathbb C}

\def \IE{\mathbb E}

\def \eps{\epsilon}

\newcommand{\IL}{\Lambda}

\newcommand{\IA}{\mathcal{A}}
\newcommand{\IB}{\mathcal{B}}

\title{Convergence analysis of explicit stabilized integrators for parabolic semilinear stochastic PDEs}

\author{
Assyr Abdulle\textsuperscript{1},
Charles-Edouard Br\'ehier\textsuperscript{2},
 and 
Gilles Vilmart\textsuperscript{3}
}

\begin{document}
\footnotetext[1]{Mathematics Section, \'Ecole Polytechnique F\'ed\'erale de Lausanne. 
The author passed away on September 1st, 2021 prior to the revision of this paper.
}
\footnotetext[2]{Univ Lyon, CNRS, Universit\'e Claude Bernard Lyon 1, UMR5208, Institut Camille Jordan, F-69622 Villeurbanne, France. brehier@math.univ-lyon1.fr}
\footnotetext[3]{Universit\'e de Gen\`eve, Section de math\'ematiques, CP 64, 1211 Gen\`eve 4, Switzerland, Gilles.Vilmart@unige.ch}
\maketitle

	\begin{abstract}
Explicit stabilized integrators are an efficient alternative to implicit or semi-implicit methods to avoid the severe timestep restriction faced by standard explicit integrators applied to stiff diffusion problems.
In this paper, we provide a fully discrete strong convergence analysis of a family of explicit stabilized methods coupled with finite element methods for a class of parabolic semilinear deterministic and stochastic partial differential equations. Numerical experiments including the semilinear stochastic heat equation with space-time white noise confirm the theoretical findings.	
	
		\smallskip
		
\end{abstract}


\section{Introduction}

In this paper, we consider semilinear parabolic stochastic partial differential equations (SPDEs), in the framework of \cite{DaPrato_Zabczyk:14}, of the form
\begin{equation}\label{eq:SPDE}
du(t)=\IL u(t)dt+F(u(t))dt+\sigma dW^Q(t),~u(0)=u_0,
\end{equation}
where $u_0$ is a given initial condition, $\IL$ is a symmetric diffusion operator, 
and $F$ is a smooth nonlinearity,
$\sigma > 0$, $\bigl(W^Q(t)\bigr)_{t\ge 0}$ is a $Q$-Wiener process.
The semi-discrete approximation obtained by spatial discretization using the finite element method with piecewise linear elements is defined by the following equation: for a spatial mesh size $h\in(0,1]$,
\begin{equation} \label{eq:SPDE_fem_intro}
du^h(t)=\IL_h u^h(t)dt+P_hF(u^h(t))dt+\sigma P_hdW^Q(t),~u^h(0)=P_hu_0,
\end{equation}
where $P_h$ denotes the $L^2$-orthogonal projection operator onto the finite element space. 
The stochastic evolution problem \eqref{eq:SPDE_fem_intro} is driven by a $Q_h$-Wiener process, with $Q_h=P_hQP_h$.

The spatial discretization $\IL_h$ of the diffusion operator $\IL$ typically yields large eigenvalues, with spectral radius of size $\bigo(h^{-2})$ for a Laplace operator and a finite element mesh with size $h$, which yields a severe timestep restriction $\tau = \bigo(h^{2}) \ll 1$ for standard explicit integrators. 
To avoid such prohibitive time-step size restrictions, a natural approach is to consider implicit or semi-implicit schemes, such as the linear implicit Euler method, for $n\geq0$,
\begin{equation}
\label{eq:meth_eul_impl}
u_{n+1}^{h}=(I-\tau\IL_h)^{-1}\bigl( u_n^h+P_h\tau F(u_n^h)+\sigma\Delta W_n^Q\bigr),
\end{equation}
where $h\in(0,1]$ and $\tau\in(0,1)$, and the Wiener increments are defined by $\Delta W_n^Q=W^Q(t_{n+1})-W^Q(t_n)$ with $t_n=n\tau$.
Alternatively to implicit methods, in this paper we consider families of explicit stabilized schemes of the form
\begin{equation}\label{eq:scheme_intro}
u_{n+1}^{h}=\IA_{s}(\tau\IL_h)u_n^h+\IB_{s}(\tau\IL_h)P_h\bigl(\tau F(u_n^h)+\sigma\Delta W_n^Q\bigr),
\end{equation}
where $\IA_s$ and $\IB_s$ are polynomials of degree $s=s(\tau,h)$ which are chosen to satisfy a suitable stability condition depending on $\tau$ and $h$. For well-chosen polynomials, the stability domain can be very large, which allow to choose the time-step size $\tau$ independently of the mesh size $h$. In particular, we shall consider variants of the SK-ROCK method from \cite{AAV18}.

Explicit stabilized integrators are well known efficient time integrators for stiff problems, in particular arising from diffusion PDE problems, both in the deterministic and stochastic settings, see \cite[Sect.\ts IV.2]{HaW96} and the review \cite{Abd13c}.
Compared to the explicit stabilized integrators first proposed in the stochastic context in \cite{AbC08,AbL08} using only first kind Chebyshev polynomials and a large damping parameter $\eta$,
a new optimal family of explicit stabilized methods
involving second kind Chebyshev polynomials,
analogous to \eqref{eq:scheme_intro}, was introduced in \cite{AAV18}. It permits the efficient integration of stiff systems of stochastic differential equations, with favourable mean-square stability properties and high-order for sampling the invariant measure of ergodic problems such as the overdamped Langevin equation. 

While the convergence analysis of stochastic explicit stabilized methods is known in finite dimension, see e.g. \cite{AbC08,AbL08,AbV13,AVZ13}, 
applied to a semi-discrete in space problem \eqref{eq:SPDE_fem_intro},
in general for explicit stabilized methods one obtains convergence estimates with error constants that depend on the spatial mesh size $h$.
The aim of this paper is to remove this dimension dependency and to provide the first strong convergence analysis of families of explicit stabilized methods in the context of semilinear parabolic SPDEs such as the stochastic heat equation, with error constants that are independent of the space and time mesh parameters $h$ and $\tau$. 
Under natural assumptions, for a fixed time interval size $T>0$, we prove that the families of explicit stabilized schemes \eqref{eq:scheme_intro} satisfy the following space-time strong error estimate. For all $\alpha \in [0,1]$ chosen small enough such that $(-\Lambda)^{(\alpha-1)/2}Q^{1/2}$ is a bounded Hilbert-Schmidt operator (corresponding to $\alpha\in[0,1/2)$ for the one-dimensional stochastic heat equation with space time-white noise), we show that the mean-square strong error satisfies
$$
\bigl(\IE|u(t_n)-u_n^h|^2\bigr)^{\frac12}\le C(1+|u_0|t_n^{-\frac{\alpha}{2}})(\tau^{\frac{\alpha}{2}}+h^\alpha),
$$
for all $\tau,h\in(0,1]$ with $t_n=n\tau\leq T$, where $C$ is independent of $n$ and both the space and time mesh parameters $h,\tau$ and the degree $s=s(\tau,h)$. The constant $C$ depends however on $T$ and $\alpha$. The parameter $\alpha$ is related to the spatial and temporal regularity of the process. 
Note that in the deterministic case ($\sigma=0$ in~\eqref{eq:SPDE}), we obtain an order one of convergence in time and order two in space, corresponding formally to $\alpha\rightarrow 2$. Here, $|\cdot|$ denotes the $L^2(\mathcal{D})$ norm in space on a bounded, open and convex polyhedral domain $\mathcal{D}$ in dimension $d$.

The literature on the strong approximation of parabolic SPDEs is wide.
We refer for instance to the incomplete list of contributions \cite{DavieGaines:01,GyongyMillet:05,GyongyNualart:97,JentzenKloeden:09_2,KloedenShott:01,LordTambue:13,Printems:01,Walsh:05,Wang:17} for various results concerning the convergence of exponential and linear implicit Euler methods applied to the SPDE~\eqref{eq:SPDE}. The novelty of this manuscript is to achieve similar strong convergence results with rates for a class of explicit stabilized methods as an alternative to implicit methods or exponential type methods.

While the scheme \eqref{eq:scheme_intro}
achieves the same strong convergence rate as the classical linear implicit Euler method \eqref{eq:meth_eul_impl}, as studied in \cite{Printems:01}, we emphasize that the analysis in this paper is not a straightforward generalization of that in \cite{Printems:01} due to lower regularization properties of the explicit numerical flow for the diffusion $\IA_{s}(\tau\IL_h)$ in \eqref{eq:scheme_intro} compared to 
the linear implicit Euler method \eqref{eq:meth_eul_impl}. 
Furthermore, the convergence analysis of an explicit method for parabolic SPDEs with space and time mesh independent constants is a new result, to the best of our knowledge. 
In addition, note that the convergence analysis is performed for a general class of explicit stabilized methods satisfying suitable regularization conditions, and it provides a unified natural framework for Runge-Kutta type schemes, including the linear implicit Euler method \eqref{eq:meth_eul_impl}, applied to semilinear parabolic SPDEs.

This paper is organized as follows. In Section \ref{sec:setting} we describe the classical Hilbert space setting and assumptions used for the analysis of SPDEs and of their numerical approximations.
Section \ref{sec:num} is devoted to the definition of the explicit stabilized method in time, coupled with a standard finite element method in space and presents our main convergence results. 
Section \ref{sec:conv} is dedicated to the strong convergence analysis of the method. Finally Section \ref{sec:numexp} is dedicated to numerical experiments which confirm the theoretical findings.

\section{Setting} \label{sec:setting}

Let $H$ be the separable Hilbert space $L^2(\mathcal{D})$, where $\mathcal{D}\subset \IR^d$ is an open and convex polyhedral bounded domain in dimension $d\in\left\{1,2,3\right\}$. The inner product in $H$ is denoted by $\langle\cdot,\cdot\rangle$. The norm in $H$ is denoted by $|\cdot|$. The operator norm on the space $\mathcal{L}(H)$ of bounded linear operators from $H$ to $H$ is denoted by $\|\cdot\|_{\mathcal{L}(H)}$. The Hilbert-Schmidt norm on the space $\mathcal{L}_2(H)$ of Hilbert-Schmidt operators on $H$ is denoted by $\|\cdot\|_{\mathcal{L}_2(H)}$.

The assumptions stated in this section are standard in the literature on SPDEs~\cite{DaPrato_Zabczyk:14} and their numerical approximations~\cite{Kruse:14},~\cite{Lord_Powell_Shardlow:14}.

\subsection{Linear operator $\IL$}

The linear operator $\IL$ is defined as the linear second-order elliptic differential operator
\[
\IL u(x)={\rm div}\bigl(a(x)\nabla u(x)\bigr),
\]
with domain $D(\IL)=H^2(\mathcal{D})\cap H_0^1(\mathcal{D})$, where we consider for simplicity homogeneous Dirichlet boundary conditions on the boundary of the domain ($u=0$ on $\partial \mathcal{D}$).
We assume that the field $a:\overline{\mathcal{D}}\to \R$ is continuous on the closure $\overline{\mathcal{D}}$ of $\mathcal{D}$, of class $\mathcal{C}^\infty$ on $\mathcal{D}$, and satisfies $\underset{x\in\overline{\mathcal{D}}}\min~a(x)>0$. The following result is standard.
\begin{pro}\label{pro:IL}
The linear operator $\IL$ is unbounded and self-adjoint. There exists a complete orthonormal system $\bigl(e_m\bigr)_{m\geq1}$ of $H$, and a non-decreasing sequence $\bigl(\lambda_m\bigr)_{m\geq1}$, such that one has $\lambda_1>0$, $\underset{m\to\infty}\lim~m^{-\frac{2}{d}}\lambda_m\in(0,\infty)$, and for all $m\geq1$,
\[
\IL e_m=-\lambda_me_m.
\]
\end{pro}
The linear operator $\IL$ generates a strongly continuous semi-group on $H$, denoted by $\bigl(e^{t\IL}\bigr)_{t\ge 0}$. For all $u\in H$ and all $t\ge 0$, one has
\[
e^{t\IL}u=\sum_{m\geq1}e^{-\lambda_mt}\langle u,e_m\rangle e_m.
\]
In addition, for any $\alpha\in[-1,1]$, the linear operator $(-\IL)^{\alpha}$ is defined as follows: for all $u\in H$, one has
\[
(-\IL)^{\alpha}u=\sum_{m\geq1}\lambda_m^{\alpha}\langle u,e_m\rangle e_m.
\]
The semi-group $\bigl(e^{t\IL}\bigr)_{t\ge 0}$ satisfies the following smoothing and regularity properties: for any $\gamma\in[0,1]$, 
\begin{equation}\label{eq:smoothing}
\underset{t\ge 0}{\sup}~t^\gamma \big\|(-\IL)^{\gamma}e^{t\IL}\big\|_{\mathcal{L}(H)}<\infty~,\quad \underset{t>0}{\sup}~t^{-\gamma}\|\bigl(e^{t\IL}-I)(-\IL)^{-\gamma}\|_{\mathcal{L}(H)}<\infty.
\end{equation}

\subsection{Nonlinear operator $F$}

The nonlinear operator $F$ in the stochastic evolution equation~\eqref{eq:SPDE} is assumed to be globally Lipschitz continuous.
\begin{ass}\label{ass:F}
The nonlinear operator $F$ is a globally Lipschitz continuous mapping from $H$ to $H$. Precisely there exists a constant $C$ such that for all $u,v\in H$,
$$
|F(u)-F(v)| \leq C |u-v|.
$$
\end{ass}
For instance, let $f:\IR\to\IR$ be a globally Lipschitz continuous real-valued function. If $F$ is defined as the associated so-called Nemytskii operator, precisely $F(u)=f(u(\cdot))$ for all $u\in H$, then Assumption~\ref{ass:F} is satisfied. The local Lipschitz continuous case is out of the scope of this article, and would require to modify the scheme (using for instance implicit methods, splitting~\cite{BG19} or tamed or truncated explicit~\cite{JH15} schemes, adaptive time-stepping~\cite{KL18} techniques).

\subsection{$Q$-Wiener process and mild solution of the SPDE~\eqref{eq:SPDE}}

Let us first introduce the $Q$-Wiener process $\bigl(W^Q(t)\bigr)_{t\ge 0}$.
Let $\bigl(\epsilon_m\bigr)_{m\geq1}$ be a complete orthonormal system of the Hilbert space $H$, and consider a bounded sequence $\bigl(q_m\bigr)_{m\geq1}$ of nonnegative real numbers. Let $Q$ and $Q^{\frac12}$ be the bounded, linear, self-adjoint operators on $H$, defined by
\[
Qu=\sum_{m\geq1}q_m \langle u, \epsilon_m\rangle {\color{blue} \epsilon_m}~,\quad Q^{\frac12}u=\sum_{m\geq1}\sqrt{q_m}\langle u,\epsilon_m\rangle \epsilon_m.
\]
\begin{definition}\label{def:WQ}
Let $\bigl(\beta_m\bigr)_{m\geq1}$ be a family of independent standard real-valued Wiener processes, $\beta_m=\bigl(\beta_m(t)\bigr)_{t\in\IR^+}$, defined on a probability space which satisfies the usual conditions.
The $Q$-Wiener process $\bigl(W^Q(t)\bigr)_{t\in\IR^+}$ is defined as follows: for all $t\ge 0$,
\[
W^Q(t)=\sum_{m\geq1}\sqrt{q_m}\beta_m(t)\epsilon_m.
\]
\end{definition}
Note that the $Q$-Wiener process $W^Q$ takes values in $H$ if and only if $\sum_{m\geq1}q_m<\infty$ (which means that $Q$ is a trace-class operator, or equivalently that $Q^{\frac12}$ is an Hilbert-Schmidt operator). Using the regularization properties~\eqref{eq:smoothing} for the semi-group $\bigl(e^{t\IL}\bigr)_{t\ge 0}$ generated by $\IL$, it is possible to define solutions of stochastic evolution equations driven by a $Q$-Wiener process without requiring that $Q$ is trace-class. More precisely, the well-posedness of the stochastic evolution equation~\eqref{eq:SPDE} is ensured by the assumption that the parameter $\overline{\alpha}$ defined in Assumption~\ref{ass:Q} below is positive. 
We refer for details to~\cite[Chap.\ts 5]{DaPrato_Zabczyk:14} in the context of linear SPDEs with additive noise and to~\cite[Sect.\ts 7.1]{DaPrato_Zabczyk:14} in the context of semilinear SPDEs.
\begin{ass}\label{ass:Q}
Assume that there exists $\alpha>0$ such that
$
\|(-\IL)^{\frac{\alpha-1}{2}}Q^{\frac12}\|_{\mathcal{L}_2(H)}<\infty
$,
and define
\[
\overline{\alpha}=\sup~\left\{\alpha\in[0,1];~\|(-\IL)^{\frac{\alpha-1}{2}}Q^{\frac12}\|_{\mathcal{L}_2(H)}<\infty\right\}.
\]
\end{ass}
In the sequel, the parameter $\overline{\alpha}$ plays a key role, indeed it determines the spatial and temporal regularity properties of the solutions, as well as strong rates of convergence of the numerical methods. We observe that $\overline \alpha \in (0,1]$ and for all $\alpha\in[0,\overline{\alpha})$, $\|(-\IL)^{\frac{\alpha-1}{2}}Q^{\frac12}\|_{\mathcal{L}_2(H)}<\infty$.
We recall that for an equation driven by space-time white noise in dimension $d=1$ ($Q=I$), one has $\overline{\alpha}=\frac12$. However, if $d\ge 2$ and $Q=I$, then $\|(-\IL)^{\frac{\alpha-1}{2}}\|_{\mathcal{L}_2(H)}^2=\sum_{m\geq1}\lambda_m^{\alpha-1}=\infty$ for all $\alpha\ge 0$ (see Proposition~\ref{pro:IL}), hence the need to consider colored noise ($Q\neq I$) if $d\ge 2$.

Under Assumption~\ref{ass:F} and~\ref{ass:Q}, for any initial condition $u_0\in H$ and any time $T\in(0,\infty)$, the stochastic evolution equation~\eqref{eq:SPDE} admits a unique global mild solution, which satisfies
\begin{equation}\label{eq:mild-SPDE}
u(t)=e^{tA}u_0+\int_{0}^{t}e^{(t-s)\IL}F(u(s))ds+\sigma\int_0^te^{(t-s)\IL}dW^Q(s)~,\quad \mbox{for all }t\in[0,T],
\end{equation}
see~\cite[Chap.\ts 2, Thm. 2.25]{Kruse:14B}, and see Proposition~\ref{pro:wellposed} (Appendix) for more details and temporal and spatial regularity estimates.

\section{Numerical methods} \label{sec:num}

The aim of this section is to define the semi-discrete and full-discrete approximations of the process $\bigl(u(t)\bigr)_{t\ge 0}$. We first describe the spatial discretization, performed using a finite element method. We then detail the fully-discrete scheme, which is the main topic of this article: the temporal discretization is performed using explicit-stabilized integrators. We provide several concrete examples of such schemes using Chebyshev polynomials, and state a general assumption for the analysis. Finally, we state our main convergence results.

\subsection{Spatial discretization: finite element method}

For every mesh size $h>0$ (without loss of generality, assume that $h\in(0,1]$), let $\mathcal{T}_h$ be a triangulation of the bounded, convex, polyhedral domain $\mathcal{D}$. Assume that the triangulation is quasi-uniform. The finite element spaces $V_h$ are finite dimensional linear subspaces of $V=H_0^1(\mathcal{D})$. Let $N_h$ denote the dimension of $V_h$. For simplicity, we consider $V_h$ to be the space of continuous functions on $\overline{\mathcal{D}}$, which are piecewise linear on $\mathcal{T}_h$ and zero at the boundary $\partial\mathcal{D}$.

For all $h\in(0,1]$, define the linear operator $\IL_h$ on $V_h$, equipped with the inner product~$\langle \cdot,\cdot\rangle$ inherited from $H=L^2(\mathcal{D})$, as follows: for all $u_h,v_h\in V_h$,
\[
\langle \IL_h u_h,v_h\rangle
  =-\int_{\mathcal{D}}a(x)\nabla u_h(x)\cdot\nabla v_h(x)dx.
\]
Observe that $\IL_h$ is a negative and self-adjoint linear operator on $V_h$, thus there exists an orthonormal basis $\bigl(e_{m,h}\bigr)_{1\le m\le N_h}$ of $V_h$ and positive real numbers $\bigl(\lambda_{m,h}\bigr)_{1\le m\le N_h}$ such that 
\[
\IL_h e_{m,h}=-\lambda_{m,h}e_{m,h},\quad 1\le m\le N_h,
\]
and $\lambda_{1,h}\le \cdots\le \lambda_{m,h}\le \lambda_{m+1,h}\le \cdots\le \lambda_{N_h,h}$. In addition, for all $h>0$, the spectral gap $\lambda_{1,h}=\underset{u_h\in V_h}\inf~\frac{\langle -\IL_hu_h,u_h\rangle}{|u_h|^2}$ of $-\IL_h$ is larger or equal than the spectral gap $\lambda_1=\underset{u\in H}\inf~\frac{\langle -\IL u,u\rangle}{|u|^2}$ of $-\IL$ (this is an immediate consequence of the inclusion $V_h\subset H$).

Define $P_h:H\to V_h$ as the $L^2$-orthogonal projection operator onto $V_h$ (for the inner product $\langle\cdot,\cdot\rangle$), for all $h\in(0,1]$. Then, for all $u\in H$, one has
\[
P_hu=\sum_{m=1}^{N_h}\langle u,e_{m,h}\rangle e_{m,h}.
\]

In the sequel, the following notation is extensively used. For any function $\phi:(-\infty,0]\to\IR$, and all $h\in(0,1]$, define the self-adjoint linear operator $\phi(\IL_h)$ as follows: for all $u_h\in V_h$,
\begin{equation}\label{eq:notation_phi}
\phi(\IL_h)u_h=\sum_{m=1}^{N_h}\phi(-\lambda_{m,h})\langle u_h,e_{m,h}\rangle e_{m,h}.
\end{equation}
In particular, the linear operators $(-\IL_h)^{\alpha}$, for all $\alpha\in[-1,1]$, and $e^{t\IL_h}$, for all $t\ge 0$, are defined by {the expression~\eqref{eq:notation_phi} with $\phi(z)=(-z)^\alpha$ and $\phi(z)=e^{tz}$ respectively}.

Let us state the conditions required for the analysis below. Those conditions are satisfied for piecewise linear finite element methods on a quasi-uniform triangulation of the bounded, convex, polygonal domain $\mathcal{D}$. We refer for instance to~\cite[Section~3.2]{Kruse:14B} and references therein for details.
\begin{ass}\label{ass:FE}
For all $\alpha\in[-\frac12,\frac12]$, there exists $C_\alpha\in(0,\infty)$ such that for all $h\in(0,1)$ and all $u_h\in V_h$,
\[
C_\alpha^{-1} |(-\IL_h)^{\alpha}u_h|\le |(-\IL)^{\alpha}u_h|\le C_\alpha |(-\IL_h)^{\alpha}u_h|,
\]
and for all $u\in D\bigl((-\IL)^\alpha\bigr)$,
\[
|(-\IL_h)^{\alpha}P_hu|\le C_\alpha |(-\IL)^{\alpha}u|.
\]
Moreover, for every $\alpha_1\in[0,1]$, $\alpha_2\in[\alpha_1,2]$, there exists $C_{\alpha_1,\alpha_2}\in(0,\infty)$ such that for all $h\in(0,1]$,
\[
\big\|(-\IL)^{\frac{\alpha_1}{2}}\bigl(I-P_h\bigr)(-\IL)^{-\frac{\alpha_2}{2}}\big\|_{\mathcal{L}(H)}\le C_{\alpha_1,\alpha_2}h^{\alpha_2-\alpha_1}.
\]
\end{ass}
It is then straightforward to prove the following result.
\begin{pro} \label{pro:32}
Let Assumptions~\ref{ass:Q} and~\ref{ass:FE} be satisfied. Then for all $\alpha\in[0,\overline{\alpha})$,
\[
\underset{h\in(0,1]}\sup~\|(-\IL_h)^{\frac{\alpha-1}{2}}P_hQ^{\frac12}\|_{\mathcal{L}_2(H)}<\infty.
\]
\end{pro}
The semi-discrete approximation of~\eqref{eq:SPDE} obtained by spatial discretization using the finite element method is defined by the following equation: for $h\in(0,1]$,
\begin{equation} \label{eq:SPDE_fem}
du^h(t)=\IL_h u^h(t)dt+P_hF(u^h(t))dt+\sigma P_hdW^Q(t)~,\quad u^h(0)=P_hu_0.
\end{equation}
This stochastic evolution equation is driven by a $Q_h$-Wiener process, with $Q_h=P_hQP_h$. The initial condition is given by $u_0^h=P_hu_0$. The problem is globally well-posed, and its unique solution $\bigl(u^h(t)\bigr)_{t\ge 0}$ takes values in the finite dimensional space $V_h$. It satisfies the following mild formulation
\begin{equation}\label{eq:mild-fe}
u^h(t)=e^{t\IL_h}u_0^h+\int_{0}^{t}e^{(t-s)\IL_h}P_hF(u^h(s))ds+\sigma \int_{0}^{t}e^{(t-s)\IL_h}P_hdW^Q(s),
\end{equation}
where we recall that the semi-group is defined by $e^{t\IL_h}u_h=\sum_{m=1}^{N_h}e^{-t\lambda_{m,h}}\langle u_h,e_{m,h}\rangle e_{m,h}$ for all $t\ge 0$ and $u_h\in V_h$ using~\eqref{eq:notation_phi}. The following smoothing and regularity properties (which are uniform in $h\in(0,1]$) are used in the sequel: for any $\gamma\in[0,1]$, 
\begin{equation}\label{eq:regul_expo}
\underset{h\in(0,1]}\sup~\underset{t\ge 0}{\sup}~t^\gamma \big\|(-\IL_h)^{\gamma}e^{t\IL_h}\big\|_{\mathcal{L}(V_h)}<\infty.
\end{equation}
and
\begin{equation}\label{eq:regul_expo_time}
\underset{h\in(0,1]}\sup~\underset{t>0}{\sup}~t^{-\gamma}\|\bigl(e^{t\IL_h}-I)(-\IL_h)^{-\gamma}\|_{\mathcal{L}(V_h)}<\infty,
\end{equation}
where $\|\cdot\|_{\mathcal{L}(V_h)}$ denotes the operator norm on $\mathcal{L}(V_h)$.

\subsection{Space-time discretization by explicit stabilized methods}

The objective of this section is to describe the temporal discretization of~\eqref{eq:SPDE_fem} using explicit-stabilized integrators. More precisely, we consider families of functions $\bigl(\IA_s\bigr)_{s\ge 1}$ and $\bigl(\IB_s\bigr)_{s\ge 1}$, indexed by the integer parameter $s\geq1$ which is used to describe the stability region of the associated integrators. 
Let $\tau>0$ denote the time-step size (without loss of generality, assume that $\tau\in(0,1]$). The integrators studied in this article are defined by
\begin{equation}\label{eq:scheme}
u_{n+1}^{h}=\IA_s(\tau\IL_h)u_n^h+\IB_s(\tau\IL_h)P_h\bigl(\tau F(u_n^h)+\sigma\Delta W_n^Q\bigr),
\end{equation}
for all $n\ge 0$, with the initial condition $u_0^h=P_hu_0=u^h(0)$. In~\eqref{eq:scheme}, the Wiener increments are defined by $\Delta W_n^Q=W^Q(t_{n+1})-W^Q(t_n)$, with $t_n=n\tau$. The self-adjoint linear operators $\IA_s(\tau\IL_h)$ and $\IB_s(\tau\IL_h)$ are defined using~\eqref{eq:notation_phi}:
\begin{equation}\label{eq:IA_IB}
\IA_s(\tau\IL_h)=\sum_{m=1}^{N_h}\IA_s(-\tau\lambda_{m,h})\langle \cdot,e_{m,h}\rangle e_{m,h}~,~\IB_s(\tau\IL_h)=\sum_{m=1}^{N_h}\IB_s(-\tau\lambda_{m,h})\langle \cdot,e_{m,h}\rangle e_{m,h}.
\end{equation}
From~\eqref{eq:scheme}, one gets a discrete-time mild formulation, similar to~\eqref{eq:mild-SPDE} and~\eqref{eq:mild-fe}: for all $n\ge 0$,
\begin{equation}\label{eq:mild_scheme}
u_n^h=\IA_s(\tau\IL_h)^{n}u_0^h+\sum_{k=0}^{n-1}\IA_s(\tau\IL_h)^{n-1-k}\IB_s(\tau\IL_h)P_h\bigl(\tau F(u_k^h)+\sigma\Delta W_k^Q\bigr).
\end{equation}
For given time-step size $\tau$ and the mesh size $h$, the parameter $s$ is chosen large enough such that the following stability condition is satisfied:
\begin{equation}\label{eq:CFL}
\tau\lambda_{N_h,h}\le L_s,
\end{equation}
where the size $L_s\in(0,\infty]$ of the stability domain depends on the choice of $\IA_s$. One of the minimal requirements (see Assumption~\ref{ass:integrator} below for the full list of conditions) is that
\begin{equation} \label{eq:Aleq1}
\underset{z\in[-L_s,0]}\sup~|\IA_s(z)|\le 1.
\end{equation}
Note that if $\IA_s(z)$ is a (non-constant) polynomial, then $|\IA_s(z)|\rightarrow \infty$ as $|z|\rightarrow \infty$, which yields that $L_s<\infty$ necessarily has a finite value, whereas $L_s=+\infty$ for the rational function $\IA(z)=(1-z)^{-1}$ related to the implicit Euler method.
Several choices are possible for the definitions of $\bigl(\IA_s\bigr)_{s\ge 1}$ and $\bigl(\IB_s\bigr)_{s\ge 1}$ and will be discussed in the next section.

\subsection{Chebyshev explicit stabilized integrators} \label{sec:twocheb}
Let us provide examples of explicit stabilized integrators inspired from the literature and which fit in this framework, and discuss implementation details.
The first choice is inspired from \cite{AAV18} in the context of stiff stochastic integrators with favourable mean-square stability properties, while the second choice in inspired from \cite{EVVZ19} in the context of deterministic partitioned stabilized Runge-Kutta methods for convex optimisation.

\begin{figure}[tb]
		\smallskip
		\centering
			\includegraphics[width=0.95\linewidth]{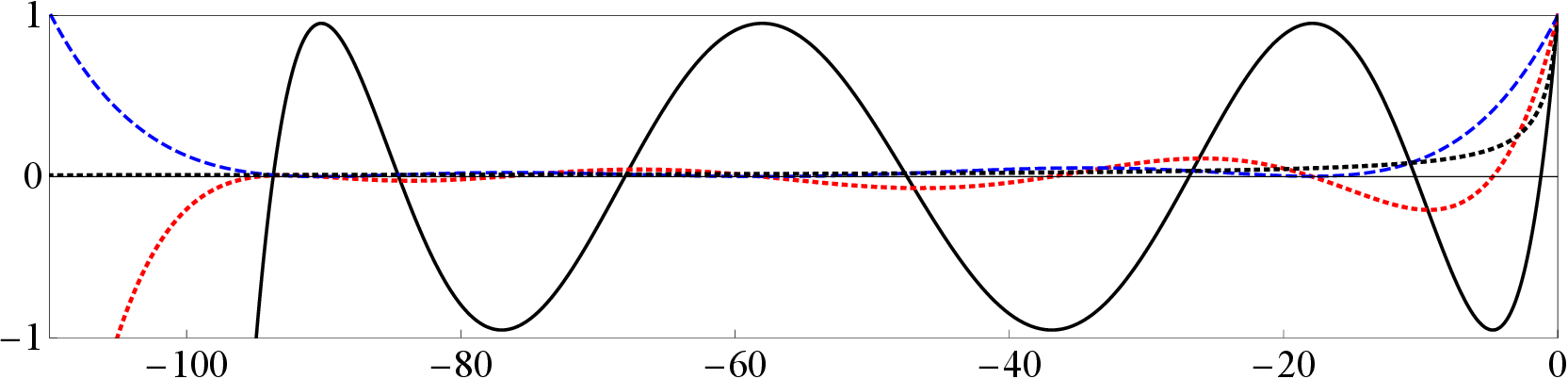}\\[2.ex]
			\includegraphics[width=0.95\linewidth]{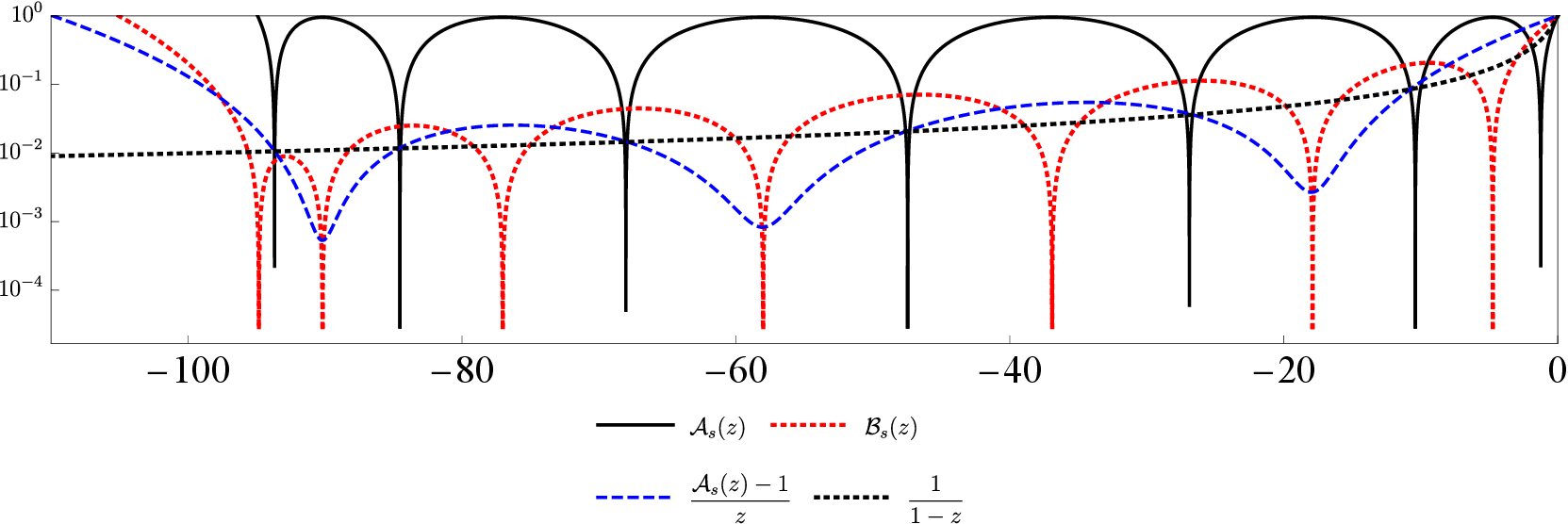}
		\caption{
			Stability functions $\IA_s(z)$ and $\IB_s(z)$ in \eqref{eq:tchebychev_scheme} of the explicit stabilized methods (normal and logarithmic scales), with degree $s=7$, damping parameter $\eta=0.05$.
			\label{fig:figstability}}
	\end{figure}

The idea of explicit stabilized methods is to consider suitably chosen polynomials $\IA_s$ such that the stability domain size $L_s$ in \eqref{eq:Aleq1} grows rapidly with $s$. 
Under the consistency condition $\IA_s(z)=1+z+\bigo(z^2)$ ($z\rightarrow 0$), it is known \cite[Sect.\ts IV.2]{HaW96} that the optimal choice maximizing the value of $L_s$ is given by $\IA_s(z)=T_s(1+z/s^2)$, where $T_s$ is the first kind Chebyshev polynomial of order $s$ satisfying $T_{s}(\cos\theta) = \cos(s\theta)$, giving $L_s=2s^2$. This quadratic growth of $L_s$ with respect to $s$ permits to achieve much larger stability domains, in contrast to standard explicit Runge-Kutta methods. However, it turns out, both in the deterministic and stochastic settings \cite{Abd13c}, that 
the condition \eqref{eq:Aleq1} is not sufficient in general for the design of a robust integrator, and some damping should be introduced to obtain a stronger bound
$$\sup_{z\in [-L_s,-\delta]} |A_s(z)| \leq 1-\eps,$$ 
where $\eps>0$ and $\delta>0$ are fixed constants independent of $s$ and the mesh sizes $\tau,h$. 
A natural choice, as considered in this article, is to consider a fixed damping parameter $\eta>0$ and define $\IA_s(z)$ as
\begin{equation}\label{eq:A_Tchebychev}
\IA_s(z)=\frac{T_s(\omega_{0}+\omega_{1}z)}{T_s(\omega_{0})},
\end{equation}
where $T_s$ is the first-kind Chebyshev polynomial of degree $s$, and the parameters $\omega_{0}, \omega_{1}$ are defined by
\begin{equation}\label{eq:tchebychev_parameters_intro}
\omega_{0}=1+\frac{\eta}{s^2},\qquad\omega_{1}=\frac{T_s(\omega_{0})}{T_s'(\omega_{0})}. 
\end{equation}
where $\eta$ is referred to as the damping parameter. The real numbers $\omega_0$ and $\omega_1$ depend on $s$ and $\eta$, but the dependence is omitted to simplify notation. 

Definition~\eqref{eq:A_Tchebychev}-\eqref{eq:tchebychev_parameters_intro} yields a stability domain with length
\[
L_s\geq\frac{2}{\omega_1}
\]
in \eqref{eq:Aleq1}, and it can be shown that
$L_s\geq (2-\frac43 \eta) s^2$, with a small and fixed damping parameter $\eta>0$. Observe that $L_s$ grows quadratically as a function of $s$, and is arbitrarily close to the optimal value $L_s=2s^2$, see~\cite{Abd13c}.

The choice for the associated function $\IB_s$ is not unique, and also not trivial to be able to take into account the noise source term, which needs to be regularized. We present two families of explicit stabilized integrators.

For the first family, as proposed in \cite{AAV18}, we use the second kind Chebyshev polynomials defined by
$T_s'(x)=sU_{s-1}(x)$ and $\sin(s\theta)=\sin(\theta)U_{s-1}(\cos\theta)$ for all $s\geq 0$. Then the SK-ROCK method \cite{AAV18} is obtained with choosing $\IA_s, \IB_s$ as follows:
for $s\ge 1$, set
\begin{equation}\label{eq:tchebychev_scheme}
\IA_s(z)=\frac{T_s(\omega_{0}+\omega_1z)}{T_s(\omega_{0})},\qquad
\IB_s(z)=\frac{U_{s-1}(\omega_{0}+\omega_{1}z)}{U_{s-1}(\omega_{0})}\bigl(1+\frac{\omega_{1}}{2}z\bigr),
\end{equation}
where $\omega_{0}, \omega_{1}$ are defined in\eqref{eq:tchebychev_parameters_intro}.
It is shown in \cite{AAV18} that the nearly optimal stability domain size $L_s\geq (2-\frac43 \eta) s^2$ also persists for the size of the mean-square stability domain of SK-ROCK in the context of stiff SDEs, taking advantage of the classical identity on first and second kind Chebyshev polynomials,
\[
T_s(x)^2+U_{s-1}(x)^2(1-x^2)=1.
\]
To illustrate the behavior of these polynomials, in Figure \ref{fig:figstability}, we plot the polynomials $\IA_s(z)$ and $\IB_s(z)$ in \eqref{eq:tchebychev_scheme} as a function of real negative $z$, and we observe that for $s=7,\eta=0.05$, the corresponding stability domain length $L_s$ in \eqref{eq:Aleq1} is close to the optimal value $2s^2=98$. In addition, the polynomial $\IA_s(z)$ oscillates between the values $\pm 1/T_s(\omega_{0})$ where $1/T_s(\omega_{0}) \simeq 1-\eta=0.95$, while $\IB_s(z)$ oscillates with a very small amplitude. We also include for comparison the polynomial $(\IA_s(z)-1)/z$ discussed below as an alternative definition of $\IB_s(z)$, and the stability function $1/(1-z)$ of the implicit Euler method.

Since $\IA_s$ and $\IB_s$ are polynomials (with degree depending on $s$), the integrators can be implemented explicitly, in particular no diagonalization or preconditioning techniques of the operator $\IL_h$ are required in practice, in constrat to implicit methods.
In practice, to avoid a dramatic accumulation of round-off errors, computing the operators $\IA_s(\tau\IL_h)$ and $\IB_s(\tau\IL_h)$ should not be performed naively as high degree polynomials (see e.g. \cite{Abd13c}).
Alternatively, taking advantage that first and second kind Chebyshev polynomials satisfy the same recursion relation: for all $s\ge 1$ 
\begin{equation*} 
T_{s+1}(x)=2xT_s(x)-T_{s-1}(x),\qquad 
U_{s+1}(x)=2xU_s(x)-U_{s-1}(x),
\end{equation*}
an efficient implementation inspired by \cite{AAV18} of the method \eqref{eq:scheme} with stability polynomials $\IA_s,\IB_s$ defined in \eqref{eq:tchebychev_scheme} can then be achieved as follows.
The first method considered in the paper then writes: given $u_n^h$, compute $u_{n+1}^h$ by induction as
\begin{equation}\label{eq:meth1}
\begin{cases}
K_{n,0}^h=u_n^h, G_n^h=P_h\bigl(\tau F(u_n^h)+\sigma\Delta W_n^Q\bigr),\\
K_{n,1}^h=K_{n,0}^h+\tau\mu_1\IL_h\bigl(K_{n,0}^h+\nu_1G_n^h\bigr)+\kappa_1G_n^h,\\
K_{n,i}^h=\mu_i\tau\IL_h K_{n,i-1}^h+\nu_iK_{n,i-1}^{h}+\kappa_i K_{n,i-2}^h,~i=2,\ldots,s,\\
u_{n+1}^h=K_{n,s}^h,
\end{cases}
\end{equation}
with $\mu_1={\omega_{1}}/{\omega_{0}}$, 
$\nu_1=s\omega_{1}/2$, $\kappa_1=s\mu_1$,
and for all $i=2,\ldots,s$,
\begin{equation} \label{eq:munukappa}
\mu_i=\frac{2\omega_1T_{i-1}(\omega_0)}{T_i(\omega_{0})},~\nu_i=\frac{2\omega_{0} T_{i-1}(\omega_{0})}{T_i(\omega_{0})},~\kappa_i=-\frac{T_{i-2}(\omega_{0})}{T_i(\omega_{0})}=1-\nu_i.
\end{equation}

We now describe the second family of explicit stabilized integrators considered in this article: the polynomials $\IA_s$ and $\IB_s$ are given by
\begin{equation}\label{eq:B_Tchebychev_V2}
\IA_s(z)=\frac{T_s(\omega_{0}+\omega_1z)}{T_s(\omega_{0})},\qquad \IB_s(z)=\frac{\IA_s(z)-1}{z},
\end{equation}
a choice considered in \cite{EVVZ19} in the context of convex optimisation. Note that the two methods use the same definition of $\IA_s$. The second method considered in this paper can be implemented as follows: given $u_n^h$, compute $u_{n+1}^h$ by induction as
\begin{equation}\label{eq:meth2}
\begin{cases}
K_{n,0}^h=u_n^h, G_n^h=P_h\bigl(\tau F(u_n^h)+\sigma\Delta W_n^Q\bigr),\\
K_{n,1}^h=K_{n,0}^h+\mu_1\bigl(\tau\IL_hK_{n,0}^h+G_n^h\bigr),\\
K_{n,i}^h=\mu_i\bigl(\tau\IL_h K_{n,i-1}^h+G_n^h\bigr)+\nu_iK_{n,i-1}^{h}+\kappa_iK_{n,i-2}^h,~i=2,\ldots,s,\\
u_{n+1}^h=K_{n,s}^h, 
\end{cases}
\end{equation}
where $\mu_1={\omega_1}/{\omega_0}$ and $\mu_i, \nu_i,\kappa_i,i=2,\ldots,s$ are defined in \eqref{eq:munukappa}.

\subsection{General assumptions and main results}

We are now in position to state the general assumptions and convergence results of this article. The analysis of convergence of the integrators~\eqref{eq:scheme} is performed under the following abstract conditions for the family of functions $\bigl(\IA_s\bigr)_{s\ge 1}$ and $\bigl(\IB_s\bigr)_{s\ge 1}$.
\begin{ass}\label{ass:integrator}
For all $s\ge 1$, $\IA_s$ and $\IB_s$ are meromorphic functions, and for some sequence $\bigl(L_s\bigr)_{s\ge 1}$ of positive real numbers, the conditions below are satisfied.
\begin{equation}\label{eq:ass_integrator:values}
\IA_s(0)=\IA_s'(0)=\IB_s(0)=1,\quad\mbox{for all } s\ge 1,
\end{equation}
In addition
\begin{equation}\label{eq:ass_integrator:bornes}
\underset{s\ge 1,z\in[-L_s,0]}{\sup}(1+|z|)|\IB_s(z)|^2<\infty,
\end{equation}
and for all $\delta\in(0,1]$,
\begin{equation}\label{eq:ass_integrator:damping}
\underset{s\ge 1,z\in[-L_s,-\delta]}{\sup}|\IA_s(z)|<1.
\end{equation}
Finally, there exists $\delta\in(0,1]$ such that
\begin{equation}\label{eq:ass_integrator:bornes_complexes}
\underset{s\ge 1,z\in\mathbb{D}(0,\delta)}{\sup}\bigl(|\IA_s(z)|+|\IB_s(z)|\bigr)<\infty,
\end{equation}
where $\mathbb{D}(0,\delta)=\{z\in\mathbb{C};~|z| < \delta\}$ denotes the open disc of radius $\delta$ and center $0$ in $\mathbb{C}$.
\end{ass}


The strong order of convergence of the explicit stabilized methods \eqref{eq:scheme} for the temporal discretization with respect to $\tau$ is equal to $\overline{\alpha}/2$, analogously to the implicit Euler method \eqref{eq:meth_eul_impl}. The strong order of convergence $\overline{\alpha}/2$ is related the H\"older regularity in time of the solution.
\begin{theorem}\label{theo:conv} \textbf{(Strong convergence in time)}
Let Assumptions~\ref{ass:Q} and~\ref{ass:integrator} be satisfied.

For all $\alpha\in[0,\overline{\alpha})$ and $T\in(0,\infty)$, there exists $C_{\alpha,T}\in(0,\infty)$ such that for any initial condition $u_0\in H$ and all $h\in(0,1]$, $\tau\in(0,1]$ and $s\ge 1$ such that the stability condition~\eqref{eq:CFL} holds, then for all $t_n=n\tau \leq T$ one has
\begin{equation}
\bigl(\IE|u^h(t_n)-u_n^h|^2\bigr)^{\frac12}\le C_{\alpha,T}(1+|u_0|t_n^{-\frac{\alpha}{2}})\tau^{\frac{\alpha}{2}}.
\end{equation}
\end{theorem}

We deduce the strong convergence rates both in time and space, taking into account the finite element spatial discretization.
This is an immediate consequence of Theorem~\ref{theo:conv}, classical finite element discretization estimates (see Proposition \ref{pro:uh} in Appendix) and the triangular inequality
$$
\bigl(\IE|u(t_n)-u_n^h|^2\bigr)^{\frac12} \leq 
\bigl(\IE|u^h(t_n)-u_n^h|^2\bigr)^{\frac12} + \bigl(\IE|u(t_n)-u^h(t_n)|^2\bigr)^{\frac12}.
$$
\begin{corollary} \label{theo:convspacetime} \textbf{(Strong convergence in time and space)}
Let Assumptions~\ref{ass:Q} and~\ref{ass:integrator} be satisfied.

For all $\alpha\in[0,\overline{\alpha})$ and $T\in(0,\infty)$, there exists $C_{\alpha,T}\in(0,\infty)$ such that for any initial condition $u_0\in H$ and all $h\in(0,1]$, $\tau\in(0,1]$ and $s\ge 1$ such that the stability condition~\eqref{eq:CFL} holds, then for all $t_n=n\tau \leq T$ one has
\begin{equation}
\bigl(\IE|u(t_n)-u_n^h|^2\bigr)^{\frac12}\le C_{\alpha,T}(1+|u_0|t_n^{-\frac{\alpha}{2}})\bigl(\tau^{\frac{\alpha}{2}}+h^\alpha\bigr),
\end{equation}
for all $t_n=n\tau \leq T$.
\end{corollary}

\begin{remark} 
The framework of our analysis includes not only explicit stabilized methods but also more general explicit or implicit Runge-Kutta type linearized methods.
Indeed, observe that the conditions stated in Assumption~\ref{ass:integrator} are satisfied if one considers the explicit Euler method (for all $s\ge 1$, set $\IA_s(z)=1+z$, $\IB_s(z)=1$), with the constraint $L_s<2$, or the implicit Euler method (for all $s\ge 1$, set $\IA_s(z)=\IB_s(z)=(1-z)^{-1}$), with $L_s=\infty$. However, note that the Crank-Nicolson method (for all $s\ge 1$, set $\IA_s(z)=\frac{1+\frac{z}{2}}{1-\frac{z}{2}}$, $\IB_s(z)=\frac{1}{1-\frac{z}{2}}$) satisfies~\eqref{eq:ass_integrator:values},~\eqref{eq:ass_integrator:bornes} and~\eqref{eq:ass_integrator:bornes_complexes}, with $L_s=\infty$, however the condition~\eqref{eq:ass_integrator:damping} is not satisfied since $|\IA_s(\infty)|=1$ (the method is $A$-stable but not $L$-stable, see~\cite{HaW96}).
\end{remark}
The convergence estimate of Theorem~\ref{theo:conv} can be generalized in several directions. The proofs would follow the same strategy as for proving Theorem~\ref{theo:conv}, with a few modifications described below and omitting some technical details.

In the case of a spatially regular noise, the following strong order one estimate can be obtained, stated for $F=0$ for simplicity (the case of a non-zero $F$ is discussed in Remark \ref{rem:nonzeroF}). Note that this corresponds formally to $\overline{\alpha}=2$ in Assumption~\ref{ass:Q} and the statement of Theorem \ref{theo:convregular}, where $\tau^{\alpha/2}$ can be replaced by $\tau^{1-\eps}$ for arbitrarily small $\eps>0$. Note that proving this result requires different techniques from the proof of Theorem~\ref{theo:conv}.

\begin{theorem}\label{theo:convregular} \textbf{(Higher-order of convergence for spatially regular additive noise)}
Under the assumptions of Theorem \ref{theo:conv}, assume in addition that
\begin{equation} \label{eq:newassump}
\underset{s\ge 1,z\in[-L_s,0]}\sup~\frac{\min(1, |z|)}{1-\IA_s(z)^2}<\infty.
\end{equation}
Consider \eqref{eq:SPDE} with $F=0$ and assume
\begin{equation}\label{eq:conditionregular}
\|(-\IL)^{\frac12}Q^{\frac12}\|_{\mathcal{L}_2}<\infty.
\end{equation}
For all $\epsilon\in(0,1)$ and $T\in(0,\infty)$, there exists $C_{\epsilon,T}\in(0,\infty)$ such that for any initial condition $u_0\in H$ and all $h\in(0,1]$, $\tau\in(0,1]$ and $s\ge 1$ such that the stability condition~\eqref{eq:CFL} holds, then for all $t_n=n\tau \leq T$ one has
\begin{equation} \label{eq:convregular}
\bigl(\IE|u^h(t_n)-u_n^h|^2\bigr)^{\frac12}\le C_{\eps,T}(1+|u_0|t_n^{-1+\eps})\tau^{1-\eps}.
\end{equation}
\end{theorem}

%

Note that Theorem~\ref{theo:conv} could be generalized to the multiplicative noise case, under appropriate assumptions. On the contrary, Theorem~\ref{theo:convregular}, where one obtains a strong order of convergence equal to $1$ only holds in the additive noise case. This is similar to the situation for the strong order of convergence of the explicit Euler-Maruyama scheme for finite-dimensional SDEs.

Finally, in the deterministic case ($\sigma=0$), the explicit-stabilized scheme has order of convergence in time equal to $1$ and order of convergence in space equal to $2$ in the deterministic case $\sigma=0$. The proof is omitted since it follows from straightforward modifications of the proof of Theorem~\ref{theo:conv}, using the H\"older regularity of the solution with order $1-\epsilon$ when $\sigma=0$.
\begin{theorem}\textbf{(Deterministic case)}
Let Assumption~\ref{ass:integrator} be satisfied.

For all $\epsilon\in(0,1)$ and $T\in(0,\infty)$, there exists $C_{\epsilon,T}\in(0,\infty)$ such that for any initial condition $u_0\in H$ and all $h\in(0,1]$, $\tau\in(0,1]$ and $s\ge 1$ such that the stability condition~\eqref{eq:CFL} holds, then for all $t_n=n\tau \leq T$ one has
\begin{equation}
|u(t_n)-u_n^h|\le C_{\epsilon,T}(1+|u_0|t_n^{-1+\epsilon})\bigl(\tau^{1-\epsilon}+h^{2-\epsilon}\bigr).
\end{equation}
\end{theorem}

\section{Convergence analysis} \label{sec:conv}

This section is dedicated to the convergence analysis of the considered class of explicit stabilized methods for semilinear parabolic SPDEs. 
We first prove that the abstract conditions (Assumption \ref{ass:integrator}) are indeed satisfied by the considered SK-ROCK method and its variant both described in Section \ref{sec:twocheb}. 
We then derive the convergence analysis, based on general consistency results, moment estimates, spatial discretization estimates.

%
%
%

%

\subsection{Verification of the abstract conditions for the Chebyshev methods}

The goal of this section is to prove Proposition~\ref{pro:AB} below, which states that the two explicit-stabilized methods \eqref{eq:meth1} and \eqref{eq:meth2}, with stability functions \eqref{eq:tchebychev_scheme} and \eqref{eq:B_Tchebychev_V2}, respectively, satisfy the conditions given in Assumption~\ref{ass:integrator}.

First, we recall some useful asymptotic results used to analyze properties of the explicit-stabilized integrators based on the first and second-kind Chebyshev polynomials, see~\cite{AAV18}. Recall that $\eta$ is the damping parameter, see~\eqref{eq:tchebychev_parameters_intro}.
\begin{pro} \label{pro:lims}
Let $\eta>0$. 
Then, when $s\to\infty$, one has
\[
\omega_{1}s^2\underset{s\to\infty}\to \frac{\tanh(\sqrt{2\eta})}{\sqrt{2\eta}}=:\Omega(\eta),
\]
and for all $\delta\in[0,\eta/\Omega(\eta)]$,
\[
T_s(\omega_{0}-\omega_{1}\delta)\underset{s\to\infty}\to \cosh\bigl(\sqrt{2(\eta-\delta \Omega(\eta))}\bigr)>1.
\]
\end{pro}
Let $\eta_0=\inf\{\eta>0;~\eta/\Omega(\eta)=1\}$. Observe that $\eta_0>0$, and that $\eta/\Omega(\eta)<1$ for all $\eta\in(0,\eta_0)$. The value of $\eta_0$ has been estimated numerically: $\eta_0\simeq 0.7001$.

We are now in position to state the main result of this section.
\begin{pro} \label{pro:AB}
Let $\bigl(\IA_s\bigr)_{s=0,1,\ldots}$ be defined by \eqref{eq:A_Tchebychev}
and
$\bigl(\IB_s\bigr)_{s=0,1,\ldots}$ be defined either by~\eqref{eq:tchebychev_scheme} or by~\eqref{eq:B_Tchebychev_V2}, with the parameters given by~\eqref{eq:tchebychev_parameters_intro}. Assume that $\eta\in(0,\eta_0)$. Then Assumption~\ref{ass:integrator} is satisfied, with $L_s=2\omega_1^{-1}$.
\end{pro}


\begin{proof}[Proof of Proposition~\ref{pro:AB}] ~\\
Note that $\IA_s$ and $\IB_s$ are polynomials, thus they are analytic functions on $\mathbb{C}$.

\noindent
{\it Proof of~\eqref{eq:ass_integrator:values}.} This follows from straightforward computations.

\noindent
{\it Proof of~\eqref{eq:ass_integrator:bornes_complexes}.} 
We recall the following formula for all $s\geq 1, k\geq 0$,
\begin{equation} \label{eq:Tsderive}
T_s^{(k)}(1)=\prod_{j=0}^{k-1}\frac{s^2-j^2}{2j+1}\le s^{2k}.
\end{equation}
Moreover, for all $x\in[1,+\infty)$, one has $T_s^{(k)}(x)\ge 0$, because all the roots of the polynomial $T_s$ (and of its derivatives $T_s^{(k)}$) belong to the interval $(-1,1)$.

Choose $\delta =1$.
Let $s\ge 1$, then $T_s$ is a polynomial of degree $s$, and the Taylor formula yields, for all $z\in \IC$,
\[
\IA_s(z)=\frac{T_s(\omega_0+\omega_1z)}{T_s(\omega_0)} = \sum_{k=0}^s \frac{1}{k! T_s(\omega_0)}T_s^{(k)}(1)(\eta/s^2 + \omega_1z)^k.
\]
Since the function $T_s$ is increasing on $[1,+\infty)$ and $\omega_0\ge 1$, one has $T_s(\omega_0)\ge T_s(1)=1$. In addition, $T_s'$ is increasing on $[1,\infty)$, thus using $T_s'(\omega_0) \geq T_s'(1)=s^2$, we obtain $\omega_1s^2\le \omega_1 T_s'(\omega)=T_s(\omega_0)$. Finally, using \eqref{eq:Tsderive} then yields 
\begin{equation} \label{eq:expeta}
\underset{s\ge 1}\sup~\omega_1s^2\le \underset{s\ge 1}\sup~
\sum_{k= 0}^s \frac{T_s^{(k)}(1)}{s^{2k}} \frac{\eta^k}{k!} 
\leq e^\eta.
\end{equation}

These properties and the estimate \eqref{eq:Tsderive} yield
\[
\underset{z\in\mathbb{D}(0,1)}{\sup}|\IA_s(z)| \leq \sum_{k=0}^s \frac{1}{k!} \bigl(\eta + \omega_1 s^2\bigr)^k \leq e^{\eta+\omega_1s^2}\leq e^{\eta+e^\eta}.
\]
To obtain an upper bound for $\underset{z\in\mathbb{D}(0,1)}{\sup}|\IB_s(z)|$, the two versions need to be treated separately. First, assume that $\IB_s$ is defined by~\eqref{eq:tchebychev_scheme}. Using the relation $T_s'=sU_{s-1}$, the inequality $T_s'(\omega_0)\ge T_s'(1)=s^2$, and the upper bound $|1+\frac{\omega_1}{2}z|\le 1+\frac{\omega_1}{2}\le 1+\frac{e^\eta}{2}\leq e^\eta$ when $z\in\mathbb{D}(0,1)$, one obtains for all $s\ge 1$
$$
\underset{z\in\mathbb{D}(0,1)}{\sup}|\IB_s(z)|\leq 
e^\eta\sum_{k=0}^{s-1} \frac{1}{k! T_s'(\omega_0)}T_s^{(k+1)}(1)(\eta/s^2 + \omega_1)^k
\leq
e^\eta \sum_{k=0}^{s-1} \frac{1}{k!} (\eta + \omega_1 s^2)^k \leq e^{2\eta+e^\eta},
$$
where we used a Taylor expansion for the polynomial $U_{s-1}(\omega_0+\omega_1z)$.

Second, assume that $\IB_s$ is defined by~\eqref{eq:B_Tchebychev_V2}. Observe that $\IB_s$ is an analytic function (since $\IA_s(0)=1$), thus using the maximum principle yields for all $s\ge 1$
\[
\underset{z\in\mathbb{D}(0,1)}{\sup}|\IB_s(z)|=\underset{|z|=1}{\sup}\Big|\frac{\IA(z)-1}{z}\Big|\le 1+\underset{z\in\mathbb{D}(0,1)}{\sup}|\IA_s(z)|\le 1+e^{\eta+e^\eta}.
\]
This concludes the proof of~\eqref{eq:ass_integrator:bornes_complexes} for the choice $\delta=1$.

\noindent
{\it Proof of~\eqref{eq:ass_integrator:damping}.} 
For a fixed $\eta$, observe that it is sufficient to prove \eqref{eq:ass_integrator:damping} for all $\delta \in(0,\eta e^{-\eta}]$. Indeed, the result then immediately follows as the left-hand side of \eqref{eq:ass_integrator:damping} is by definition a decreasing function of $\delta$.
We now assume $\delta \leq \eta e^{-\eta}$. Using \eqref{eq:expeta} yields $\delta \leq \eta/(\omega_1s^2) = (\omega_0-1)/\omega_1$.
On the one hand, if $z\in [-L_s, (1-\omega_0)/\omega_1]$, then $x=\omega_0+\omega_1z \in [-1,1]$, hence $|T_s(\omega_0+\omega_1z)|\leq 1$. On the other hand, if $z\in[(1-\omega_0)/\omega_1,-\delta]$, then $x=\omega_0+\omega_1z\ge 1$, hence $0\le T_s(\omega_0+\omega_1z)\le T_s(\omega_0-\omega_1\delta)<T_s(\omega_0)$, since $T_s$ is increasing on $[1,+\infty)$. Thus, 
$$
\underset{z\in[-L_s,-\delta]}{\sup}|\IA_s(z)| \leq \frac{\max(1,T_s(\omega_0-\omega_1 \delta))}{T_s(\omega_0)}=
\IA_s(-\delta)
\in(0,1),\quad \mbox{for all } s\geq 1. 
$$
Moreover, using Proposition \ref{pro:lims},
\[
\IA_s(-\delta)
\underset{s\to\infty}\to \frac{\cosh\bigl(\sqrt{2(\eta-\delta \Omega(\eta))}\bigr)}{\cosh\bigl(\sqrt{2\eta}\bigr)}\in(0,1).
\]
One thus obtains the inequality $\underset{s\ge 1}{\sup} 
\IA_s(-\delta)
\in (0,1)$, which concludes the proof of~\eqref{eq:ass_integrator:damping}.

\noindent
{\it Proof of~\eqref{eq:ass_integrator:bornes}.} 
Owing to~\eqref{eq:ass_integrator:bornes_complexes}, which holds with $\delta=1$, it is sufficient to prove
\[
\underset{s\ge 1,z\in[-L_s,-1]}{\sup}|z||\IB_s(z)|^2<\infty.
\]

It is again necessary to treat separately the two versions of the integrator.

First, assume that $\IB_s$ is defined by~\eqref{eq:tchebychev_scheme}. Recall that $L_s=2\omega_1^{-1}$, thus if $z\in[-2\omega_1^{-1},-1]$, then $x=\omega_0+\omega_1 z \in[-1+\eta/s^2,\omega_0-\omega_1]$. Define
$$
Q_s(x)= \frac{s^2\omega_1}{T_s(\omega_0)^2} \left(\frac{1+x-\eta/s^2}2\right)\frac{1-(x-\eta/s^2)^2}{1-x^2},
$$
then a straightforward computation implies the equality
$$
-2z\IB_s(z)^2 = U_{s-1}(x)^2(1-x^2) Q_s(x).
$$
Let us first prove the following claim: for all $x\in [-1+\eta/s^2,\omega_0-\omega_1]$, one has $Q_s(x)\ge 0$, and
\begin{equation} \label{eq:claimQ}
\underset{s\ge 1}{\sup}\underset{x\in[-1+\eta/s^2,\omega_0-\omega_1]}{\sup}Q_s(x)<\infty.
\end{equation}
Indeed, owing to \eqref{eq:expeta},
$
\frac{s^2\omega_1}{T_s(\omega_0)^2} \leq e^\eta.
$
In addition, for all $z\in[-2\omega_1^{-1},-1]$,
$$\frac{1+x-\eta/s^2}2=1+\frac{\omega_1 z}{2}\in[0,1].$$ 
To treat the last term, note that for all $x\in\IR$,
\[
\frac{d}{dx}\left(\frac{1-(x-\eta/s^2)^2}{1-x^2}\right)=
2\eta\frac{1+x^2-\eta/s^2 x}{s^2(1-x^2)^2}
\ge 0,
\]
where we used $\eta/s^2 \leq \eta_0 \leq 2$. As a consequence, for $-1+\frac{\eta}{s^2}\le x\le \omega_0-\omega_1=1+\frac{\eta}{s^2}-\omega_1$, one obtains for all $s\ge 1$
\[
0\le Q_s(x) \leq e^\eta \frac{1-(x-\eta/s^2)^2}{1-x^2}\le\underset{s\ge 1}\sup~ e^\eta  \frac{1-(1-\omega_1)^2}{1-(\omega_0-\omega_1)^2}<\infty.
\]
Here, we used the property
\[
e^\eta  \frac{1-(1-\omega_1)^2}{1-(\omega_0-\omega_1)^2}\underset{s\to\infty}\to \frac{e^\eta }{1-\eta/\Omega(\eta)},
\]
and an asymptotic expansion with $\omega_1\underset{s\to\infty}\sim \Omega(\eta)s^{-2}$ (owing to Proposition~\ref{pro:lims}), $\omega_0=1+\frac{\eta}{s^2}$, and the property $\eta/\Omega(\eta)<1$ for $\eta<\eta_0$. This concludes the proof of~\eqref{eq:claimQ}.

Since $Q_s(x)\ge 0$ for $x\in [-1+\eta/s^2,\omega_0-\omega_1]$, and since $z\le 0$, one obtains that $U_{s-1}(x)^2(1-x^2)\ge 0$. In addition, properties of the first and second-kind Chebyshev polynomials yield the equality $U_{s-1}(x)^2(1-x^2) =1-T_s(x)^2\le 1$ for all $x\in\R$. As a consequence,
\[
\underset{s\ge 1}{\sup}\underset{z\in[-2\omega_1^{-1},-1]}{\sup}-z\IB_s(z)^2\le \underset{s\ge 1}{\sup}\underset{x\in[-1+\eta/s^2,\omega_0-\omega_1]}{\sup}Q_s(x)<\infty.
\]
This concludes the proof of \eqref{eq:ass_integrator:bornes} for the first version~\eqref{eq:tchebychev_scheme}.

Let us now focus on the second version: assume that $\IB_s$ is defined by~\eqref{eq:B_Tchebychev_V2}. Then,
using $(\IA_s(z)-1)^2 \leq 2+2 \IA_s(z)^2$,
\[
\underset{s\ge 1}{\sup}\underset{z\in[-2\omega_1^{-1},-1]}{\sup}-z\IB_s(z)^2=\underset{s\ge 1}{\sup}\underset{z\in[-2\omega_1^{-1},-1]}{\sup}\frac{|\IA_s(z)-1|^2}{|z|}\le 2+2\underset{s\ge 1}{\sup}\underset{z\in[-2\omega_1^{-1},-1]}{\sup}|\IA_s(z)|^2<\infty,
\]
owing to~\eqref{eq:ass_integrator:damping} with $\delta=1$.

This concludes the proof of~\eqref{eq:ass_integrator:bornes} for the two versions.
The proof of Proposition \ref{pro:AB} is thus completed.
\end{proof}

Finally, we verify that the considered Chebyshev methods satisfy the additional assumption \eqref{eq:newassump} needed in Theorem \ref{theo:convregular}.
\begin{pro} \label{pro:ABbis}
Let $\bigl(\IA_s\bigr)_{s=0,1,\ldots}$ be defined by~\eqref{eq:A_Tchebychev} with $\eta\in(0,\eta_0)$. Then \eqref{eq:newassump} holds.
\end{pro}
\begin{proof}
Let 
$\delta\in(0,\eta e^{-\eta})$. First, \eqref{eq:ass_integrator:damping} yields
\[
\underset{s\ge 1,z\in[-L_s,-\delta]}\sup~\frac{\min(1, |z|)}{1-\IA_s(z)^2}\le \frac{1}{1-\underset{s\ge 1}\sup~\underset{z\in[-L_s,-\delta]}\sup~|\IA_s(z)|^2}<\infty.
\]
Second, let $z\in[-\delta,0]$ and $x=\omega_0+\omega_1z$.
Using
$\delta\leq \eta e^{-\eta} \leq\underset{s\ge 1}\inf~\frac{\omega_0-1}{\omega_1}$ (owing to \eqref{eq:expeta}),
 then one has $x\ge \omega_0-\omega_1\delta \ge 1$. Thus, $\IA_s(z)$ and $\IA_s'(z)$ are positive and increasing on the interval $[-\delta,0]$. 
In addition, using
\begin{equation*}
\frac{\IA_s(z)-1}{z}=\int_0^1\IA_s'(tz)dt
\geq\int_0^1\IA_s'(-t\delta)dt
=\frac{1-\IA_s(-\delta)}{\delta},
\end{equation*}
and noting $1-\IA_s(z)^2\geq 1-\IA_s(z)$, we obtain
\[
\underset{s\ge 1,z\in[-\delta,0]}\sup~\frac{\min(1, |z|)}{1-\IA_s(z)^2}
\leq
\frac{\delta}{1-\underset{s\ge 1}\sup~\IA_s(-\delta)}<\infty,
\]
where we used \eqref{eq:ass_integrator:damping}. This concludes the proof of Proposition~\ref{pro:ABbis}.
\end{proof}

\subsection{General consistency analysis}

The objective of this section is to state and prove some consequences of the abstract conditions given in Assumption~\ref{ass:integrator}. We start with the following lemma, which states that $\IB_s(z)=1+\bigo(z)$ and $\IA_s(z)^n=e^{nz}+\bigo(n^{-1})$, uniformly for $z\in[-L_s,0]$ and $s\ge 1$.
\begin{lemma}\label{lem:integrator}
Let Assumption~\ref{ass:integrator} be satisfied. Then
\begin{equation}\label{eq:lem_integrator:B}
\underset{s\ge 1,z\in[-L_s,0]}{\sup}\frac{|1-\IB_s(z)|}{|z|}<\infty,
\end{equation}
and
\begin{equation}\label{eq:lem_integrator:A-expo}
\underset{n\ge 1,s\ge 1,z\in[-L_s,0]}{\sup}n\big|\IA_s(z)^n-e^{nz}\big|<\infty.
\end{equation}
\end{lemma}
The estimate \eqref{eq:lem_integrator:A-expo} is inspired from \cite[Theorem 9]{crouzeix05aop} for the time discretization of parabolic PDEs and from \cite[Lemma 5.2]{AbV12} in the context of explicit stabilized methods
for parabolic homogenization problems.

\begin{proof}[Proof of Lemma \ref{lem:integrator}]
Let $r_s(z)={(1-\IB_s(z))}/{z}$. Since $\IB_s$ is a meromorphic function (by Assumption~\ref{ass:integrator}) with $\IB_s(0)=1$ (owing to~\eqref{eq:ass_integrator:values}) and which is bounded on $\mathbb{D}(0,\delta)$, then $r_s$ is an analytic function in the open disc $\mathbb{D}(0,\delta)$, where $\delta$ is given by~\eqref{eq:ass_integrator:bornes_complexes} (recall that $\delta$ does not depend on $s$).
Owing to the maximum principle for analytic functions,
\[
\underset{|z|\le \delta}\sup~|r_s(z)|=\underset{\theta\in[0,2\pi]}\max |r_s(\delta e^{i\theta})|\le \frac{C}{\delta}
\]
In addition, using~\eqref{eq:ass_integrator:bornes}, we deduce
\[
\underset{z\in[-L_s,-\delta]}\sup~|r_s(z)|
\le \delta^{-1} \bigl(1+\underset{z\in[-L_s,-\delta]}\sup~|\IB_s(z)|\bigr)<\infty.
\]
This concludes the proof of~\eqref{eq:lem_integrator:B}.


It remains to prove~\eqref{eq:lem_integrator:A-expo}. For $n\geq1$ and $s\ge 1$, set $\varphi_{n,s}(z) = \IA_s(z)^n-e^{nz}$, and let $r_{1,s}(z)={\varphi_{1,s}(z)}/{z^2}$. Since $\IA_s$, and thus $\varphi_{1,s}$, are meromorphic functions (by Assumption~\ref{ass:integrator}), with $\varphi_{1,s}(0)=\varphi_{1,s}'(0)=0$ (owing to~\eqref{eq:ass_integrator:values}), then $r_{1,s}$ is an analytic function in the open disc $\mathbb{D}(0,\delta)$, where $\delta$ is given by~\eqref{eq:ass_integrator:bornes_complexes} (again, $\delta$ does not depend on $s$).
Owing to the maximum principle for analytic functions, one has
\[
C_0=\underset{s\ge 1,|z|\le \delta}{\sup}|r_{1,s}(z)|=\underset{s\ge 1}{\sup}\underset{|z|=\delta}{\sup}|r_{1,s}(z)|\le \frac{1}{\delta^2}\underset{s\ge 1,|z|\le \delta}{\sup}(e^{\delta}+|\IA_s(z)|)<\infty.
\]
Let $\gamma\in(0,\delta)$ be chosen sufficiently small such that $\frac{\gamma}{8}+C_0\gamma e^{\gamma/2}<\frac12$. Then, for all $z\in[-\gamma,0]$,
\[
|\IA_s(z)| = e^{z/2} |e^{z/2}-e^{-z/2} \varphi_{1,s}(z)| \leq 
e^{z/2} (1-|z|/2+|z|^2/8 + C_0|z|^2 e^{-z/2}) \leq  e^{z/2}.
\]
Using the identity
$$
|\varphi_{n,s}(z)|=|\IA_s(z)^n-e^{nz}|=|z^{-2}\varphi_{1,s}(z)| |z^2\sum_{k=0}^{n-1} e^{kz} \IA(z)^{n-k-1}|,
$$
one obtains for all $n\geq 1$,
$$
\underset{s\ge 1,z\in[-\gamma,0]}{\sup}|\varphi_{n,s}(z)| \le C_0 |z|^2 ne^{(n-1)z/2} \le 
\frac{1}{n}C_0e^{\gamma/2}\underset{x\ge 0}{\sup}\bigl(x^2e^{-x/2}\bigr).
$$
Let $\rho=\underset{s\ge 1,z\in[-L_s,-\gamma]}{\sup}|\IA_s(z)|$. Then $\rho<1$ owing to~\eqref{eq:ass_integrator:damping}, and it is straightforward to check that, for all $n\ge 1$, one has
$$
\underset{n\ge 1, s\ge 1,z\in[-L_s,-\gamma]}{\sup}n|\varphi_{n,s}(z)|\leq \underset{n\ge 1}\sup~n\bigl(e^{-n\gamma}+\rho^n\bigr)<\infty. 
$$
Gathering the upper bounds for $z\in[-L_s,-\gamma]$ and $z\in [-\gamma,0]$ concludes the proof of~\eqref{eq:lem_integrator:A-expo}. The proof of Lemma~\ref{lem:integrator} is thus completed.
%
\end{proof}

Let us now apply the conditions given in Assumption~\ref{ass:integrator} and obtained in Lemma~\ref{lem:integrator}, to deduce properties for the operators $\IA_s(\tau\IL_h)$ and $\IB_s(\tau\IL_h)$. Recall that $\|\cdot\|_{\mathcal{L}(V_h)}$ denotes the operator norm on $\mathcal{L}(V_h)$.
\begin{pro}\label{pro:integrator}
Let Assumption~\ref{ass:integrator} be satisfied. For all $\gamma\in[0,1]$, there exists $C \in (0,\infty)$ such that for all $\tau,h$, and for $s\geq 1$ satisfying the stability condition~\eqref{eq:CFL}, then
\begin{align}
\|\IA_s(\tau\IL_h)\|_{\mathcal{L}(V_h)}&\le 1,\label{eq:pro_integrator:A}\\
\|(-\IL_h)^{\frac{\gamma}{2}}\IB_s(\tau\IL_h)\|_{\mathcal{L}(V_h)}&\le \frac{C}{\tau^{\frac{\gamma}{2}}},\label{eq:pro_integrator:B}\\
\|\bigl(I-\IB_s(\tau\IL_h)\bigr)(-\IL_h)^{-\gamma}\|_{\mathcal{L}(V_h)}&\le C \tau^\gamma,\label{eq:pro_integrator:B_error}\\
\|\IA_s(\tau\IL_h)^n-e^{n\tau\IL_h}\|_{\mathcal{L}(V_h)}&\le \frac{C}{n},\quad\mbox{for all }n\geq 1.\label{eq:pro_integrator:A_expo}
\end{align}
\end{pro}

\begin{proof}[Proof]
The linear operators $\IA_s(\tau\IL_h)$, $\IB_s(\tau\IL_h)$, $(-\IL_h)^\gamma$, $(-\IL_h)^{-\gamma}$, and $e^{n\tau\IL_h}$ are self-adjoint operators, recall the notation~\eqref{eq:notation_phi}. Owing to the stability condition~\eqref{eq:CFL}, one has $-\tau\lambda_{m,h}\in[-L_s,0]$ for all $m\in\{1,\ldots,N_h\}$. Inequality~\eqref{eq:pro_integrator:A} is then a straightforward consequence of~\eqref{eq:ass_integrator:damping} (which holds for all arbitrarily small $\delta>0$). Inequality~\eqref{eq:pro_integrator:B} is a consequence of~\eqref{eq:ass_integrator:bornes} (in the cases $\gamma=0$ and $\gamma=1$), and of the following interpolation argument: using the inequalities $\IB_s(z)^2\le C_0$ and $|z|\IB_s(z)^2\le C_1$, one obtains $|z|^\gamma\IB_s(z)^2\le C_0^{1-\gamma}C_1^{\gamma}$, for $\gamma\in[0,1]$ and $z=-\tau\lambda_{m,h}$. Similarly, Inequality~\eqref{eq:pro_integrator:B_error} is a consequence of~\eqref{eq:lem_integrator:B} (in the case $\gamma=1$). Finally, Inequality~\eqref{eq:pro_integrator:A_expo} follows from~\eqref{eq:lem_integrator:A-expo}.
\end{proof}



\subsection{Proof of the main convergence estimates}

We have obtained in the previous section all the ingredients for the proof of Theorem~\ref{theo:conv}. The structure of the proof follows the same approach as for the convergence of the linear implicit Euler scheme applied to parabolic semilinear SPDE discretization analysis with Lipschitz continuous nonlinearities, see for instance~\cite{Printems:01}. Our proof however illustrates the behavior of explicit-stabilized integrators in this context.

In the proofs, the value of the constant $C\in(0,\infty)$ may change from line to line.

The following estimate on the moments of the fully-discrete scheme is a key ingredient for the proof of the convergence estimates.
\begin{pro}\label{pro:scheme}
Let Assumptions~\ref{ass:F},~\ref{ass:Q} and~\ref{ass:FE} be satisfied.

Let $\bigl(u_n^h\bigr)_{n\ge 0}$ be defined by~\eqref{eq:scheme}, where $\bigl(\IA_s\bigr)_{s\ge 1}$ and $\bigl(\IB_s\bigr)_{s\ge 1}$ satisfy Assumption~\ref{ass:integrator}.

For all $\alpha\in[0,\overline{\alpha})$ and $T\in(0,\infty)$, there exists $C_{\alpha,T}\in(0,\infty)$ such that for all $u_0\in D((-\IL)^\frac{\alpha}{2})$, and all $h\in(0,1]$, $\tau\in(0,1]$ and $s\ge 1$ such that the stability condition~\eqref{eq:CFL} is satisfied, one has
\[
\underset{n\tau\le T}\max~\IE|(-\IL_h)^{\frac{\alpha}{2}}u_n^h|^2\le C_{\alpha,T}(1+|(-\IL)^{\frac{\alpha}{2}}u_0|^2).
\]
\end{pro}
In the proof of Theorem~\ref{theo:conv}, only the result with $\alpha=0$ is used. Note that Proposition~\ref{pro:scheme} shows that the spatial regularity is preserved by the temporal discretization, uniformly in the admissible parameters $h,\tau,s$.

\begin{proof}[Proof of Proposition~\ref{pro:scheme}]
By the Minkowski inequality, owing to the formulation~\eqref{eq:mild_scheme} of the fully-discrete scheme, it is sufficient to deal with the three following contributions.
\begin{enumerate}
\item[(i)] Owing to~\eqref{eq:pro_integrator:A} and using Assumption~\ref{ass:FE}, one has
\[
|(-\IL_h)^{\frac{\alpha}{2}}\IA_s(\tau \IL_h)^n u_0^h|\le |(-\IL_h)^{\frac{\alpha}{2}}P_h u_0|\le C |(-\IL)^{\frac{\alpha}{2}}u_0|.
\]
\item[(ii)] Owing to the Minkowski inequality, and using the linear growth $|F(x)|\leq C(1+|x|)$ (a consequence of the Lipschitz continuity of $F$), one has
\begin{align*}
\Bigl(\IE\big|&(-\IL_h)^{\frac{\alpha}{2}}\tau\sum_{k=0}^{n-1}\IA_s(\tau \IL_h)^{n-1-k}\IB_s(\tau \IL_h)P_hF(u_k^h)\big|^2\Bigr)^{\frac12}\\
&\le \tau\sum_{k=0}^{n-1}\Bigl(\IE\big|(-\IL_h)^{\frac{\alpha}{2}}\IA_s(\tau \IL_h)^{n-1-k}\IB_s(\tau \IL_h)P_hF(u_k^h)\big|^2\Bigr)^{\frac12}\\
&\le C \tau\sum_{k=0}^{n-1}\big\|(-\IL_h)^{\frac{\alpha}{2}}\IA_s(\tau \IL_h)^{n-1-k}\IB_s(\tau \IL_h)\|_{\mathcal{L}(V_h)}\bigl(1+\bigl(\IE|u_k^h|^2\bigr)^{\frac12}\bigr)\\
&\le C \tau\sum_{k=0}^{n-1}\big\|(-\IL_h)^{\frac{\alpha}{2}}e^{(n-1-k)\tau \IL_h}\IB_s(\tau \IL_h)\big\|_{\mathcal{L}(V_h)}\bigl(1+\bigl(\IE|u_k^h|^2\bigr)^{\frac12}\bigr)\\
&+C\tau\sum_{k=0}^{n-1}\big\|\bigl(\IA_s(\tau \IL_h)^{n-1-k}-e^{(n-1-k)\tau \IL_h}\bigr)(-\IL_h)^{\frac{\alpha}{2}}\IB_s(\tau \IL_h)\|_{\mathcal{L}(V_h)}\bigl(1+\bigl(\IE|u_k^h|^2\bigr)^{\frac12}\bigr).
\end{align*}

On the one hand, owing to~\eqref{eq:pro_integrator:B} (applied with $\gamma=\alpha$), and to~\eqref{eq:regul_expo}, one has
\begin{align*}
\tau\sum_{k=0}^{n-1}&\big\|(-\IL_h)^{\frac{\alpha}{2}}e^{(n-1-k)\tau \IL_h}\IB_s(\tau \IL_h)\big\|_{\mathcal{L}(V_h)}\bigl(1+\bigl(\IE|u_k^h|^2\bigr)^{\frac12}\bigr)\\
&=\tau\big\|(-\IL_h)^{\frac{\alpha}{2}}\IB_s(\tau \IL_h)\big\|_{\mathcal{L}(V_h)}\bigl(1+\bigl(\IE|u_{n-1}^h|^2\bigr)^{\frac12}\bigr)\\
&~+\tau\sum_{k=0}^{n-2}\big\|(-\IL_h)^{\frac{\alpha}{2}}e^{(n-1-k)\tau \IL_h}\IB_s(\tau \IL_h)\big\|_{\mathcal{L}(V_h)}\bigl(1+\bigl(\IE|u_k^h|^2\bigr)^{\frac12}\bigr)\\
&\le C\tau^{1-\frac{\alpha}{2}}\bigl(1+\bigl(\IE|u_{n-1}^h|^2\bigr)^{\frac12}\bigr)\\
&~+C\tau\sum_{k=0}^{n-2}\frac{1}{\bigl((n-1-k)\tau\bigr)^{\frac{\alpha}{2}}}\bigl(1+\bigl(\IE|u_k^h|^2\bigr)^{\frac12}\bigr).
\end{align*}

On the other hand, using first~\eqref{eq:pro_integrator:A_expo}, then~\eqref{eq:pro_integrator:B},
\begin{align*}
\tau\sum_{k=0}^{n-1}&\big\|\bigl(\IA_s(\tau \IL_h)^{n-1-k}-e^{(n-1-k)\tau \IL_h}\bigr)(-\IL_h)^{\frac{\alpha}{2}}\IB_s(\tau \IL_h)\|_{\mathcal{L}(V_h)}\bigl(1+\bigl(\IE|u_k^h|^2\bigr)^{\frac12}\bigr)\\
&\le C\tau\sum_{k=0}^{n-2}\frac{1}{(n-1-k)
}\big\|(-\IL_h)^{\frac{\alpha}{2}}
\IB_s(\tau \IL_h)\|_{\mathcal{L}(V_h)}\bigl(1+\bigl(\IE|u_k^h|^2\bigr)^{\frac12}\bigr)\\
&\le C\tau\sum_{k=0}^{n-2}\frac{1}{\bigl((n-1-k)\tau\bigr)^{\frac{\alpha}{2}}
}\bigl(1+\bigl(\IE|u_k^h|^2\bigr)^{\frac12}\bigr).
\end{align*}


\item[(iii)] Using the It\^o isometry formula, the condition $\big\|(-\IL_h)^{\frac{\alpha+\epsilon-1}{2}}P_hQ^{\frac12}\|_{\mathcal{L}_2(H)}\le C$, with $\epsilon\in(0,\overline{\alpha}-\alpha)$ (see Proposition \ref{pro:32}), and the inequalities~\eqref{eq:regul_expo},~\eqref{eq:pro_integrator:B} and~\eqref{eq:pro_integrator:A_expo}, combined with the same arguments as above, one obtains
\begin{align*}
\IE\big|(-\IL_h)^{\frac{\alpha}{2}}&\sum_{k=0}^{n-1}\IA_s(\tau \IL_h)^{n-1-k}\IB_s(\tau\IL_h)P_h \Delta W_k^Q|^2\\
&=\tau\sum_{k=0}^{n-1}\big\|(-\IL_h)^{\frac{\alpha}{2}}\IA_s(\tau\IL_h)^{n-1-k}\IB_s(\tau\IL_h)P_hQ^{\frac12}\big\|_{\mathcal{L}_2(H)}^{2}\\
&\le \tau\sum_{k=0}^{n-1}\big\|(-\IL_h)^{\frac{1-\epsilon}{2}}\IA_s(\tau\IL_h)^{n-1-k}\IB_s(\tau\IL_h)\big\|_{\mathcal{L}(V_h)}^{2}\big\|(-\IL_h)^{\frac{\alpha+\epsilon-1}{2}}P_hQ^{\frac12}\|_{\mathcal{L}_2(H)}^2\\
&\le C\tau\sum_{k=0}^{n-1}\big\|(-\IL_h)^{\frac{1-\epsilon}{2}}e^{(n-1-k)\tau\IL_h}\IB_s(\tau\IL_h)\big\|_{\mathcal{L}(V_h)}^{2}\\
&+C\tau\sum_{k=0}^{n-1}\big\|(-\IL_h)^{\frac{1-\epsilon}{2}}\bigl(\IA_s(\tau\IL_h)^{n-1-k}-e^{(n-1-k)\tau)\IL_h}\bigr)\IB_s(\tau\IL_h)\big\|_{\mathcal{L}(V_h)}^{2}\\
&\le C\bigl(\tau^{\epsilon}+\tau\sum_{k=0}^{n-2}\frac{1}{\bigl((n-1-k)\tau\bigr)^{1-\epsilon}}\bigr)\\
&+C\tau \sum_{k=0}^{n-2}\frac{1}{\bigl((n-1-k)\tau\bigr)^{1-\epsilon}}\big\|(-\tau \IL_h)^{\frac{1-\epsilon}{2}}\IB_s(\tau \IL_h)\big\|_{\mathcal{L}(V_h)}^2\\
&\le C.
\end{align*}
\end{enumerate}
To conclude, first assume that $\alpha=0$. The estimate then follows from the application of the discrete Gronwall lemma. The case $\alpha\in(0,\overline{\alpha})$ then follows from the calculations above.
This concludes the proof of Proposition~\ref{pro:scheme}. 
\end{proof}

\begin{proof}[Proof of Theorem~\ref{theo:conv}]

Introduce the notation $t_n=n\tau$, $n\in\mathbb{N}_0$, and assume that $T=N\tau$, with $N\in\mathbb{N}$. Set $F_h=P_hF$.

Let also $\epsilon_n=\bigl(\IE|u^h(n\tau)-u_n^h|^2\bigr)^{\frac12}$. Using the mild formulations~\eqref{eq:mild-fe} and~\eqref{eq:mild_scheme}, one obtains the decomposition
\begin{align*}
u^h(t_n)-u_n^h&=\bigl(e^{n\tau\IL_h}-\IA_s(\tau\IL_h)^n\bigr)u_0^h\\
&+\sum_{k=0}^{n-1}\int_{t_k}^{t_{k+1}}\bigl[e^{(t_n-t)\IL_h}-\IA_s(\tau\IL_h)^{n-1-k}\IB_s(\tau\IL_h)\bigr]P_h dW^Q(t)\\
&+\sum_{k=0}^{n-1}\int_{t_k}^{t_{k+1}}\bigl[e^{(t_n-t)\IL_h}P_hF(u^h(t))-\IA_s(\tau\IL_h)^{n-1-k}\IB_s(\tau\IL_h)F_h(u_k^h)\bigr]dt.
\end{align*}
Using Minkowski inequality, one obtains $\epsilon_n\le \epsilon_n^1+\epsilon_n^2+\epsilon_n^3$, where
\begin{align*}
\epsilon_n^1&=\big|\bigl(e^{n\tau\IL_h}-\IA_s(\tau\IL_h)^n\bigr)u_0^h\big|\\
\epsilon_n^2&=\Bigl(\Big|\sum_{k=0}^{n-1}\IE\int_{t_k}^{t_{k+1}}\Bigl[e^{(t_n-t)\IL_h}-\IA_s(\tau\IL_h)^{n-1-k}\IB_s(\tau\IL_h)\Bigr]P_h dW^Q(t)\Big|^2\Bigr)^{\frac{1}{2}}\\
\epsilon_n^3&=\sum_{k=0}^{n-1}\int_{t_k}^{t_{k+1}}\bigl(\IE\big|e^{(t_n-t)\IL_h}F_h(u^h(t))-\IA_s(\tau\IL_h)^{n-1-k}\IB_s(\tau\IL_h)F_h(u_k^h)\big|^2\bigr)^{\frac12} dt.
\end{align*}
It remains to prove the three claims below: for all $\alpha\in[0,\overline{\alpha})$, there exists  a constant $C$ such that
\begin{eqnarray}
\epsilon_n^1&\le & \frac{C|u_0|\tau^{\frac{\alpha}{2}}}{t_n^{\frac{\alpha}{2}}},\label{eq:claim1}\\
\epsilon_n^2&\le & C \tau^{\frac{\alpha}{2}},\label{eq:claim2}\\
\epsilon_n^3&\le & C\tau \sum_{k=0}^{n-1}\epsilon_{k}+C\tau^{\frac{\alpha}{2}}(1+|u_0|).\label{eq:claim3}
\end{eqnarray}
{\it Proof of~\eqref{eq:claim1}.}
This claim follows from~\eqref{eq:pro_integrator:A_expo}, indeed this inequality yields for all $n\ge 1$
\[
|\epsilon_n^1|\le \frac{C}{n}|u_0^h|\le \frac{C|u_0|\tau^{\frac{\alpha}{2}}}{(n\tau)^{\frac{\alpha}{2}}}.
\]
{\it Proof of~\eqref{eq:claim2}.}
Using the It\^o isometry formula,
\begin{align*}
(\epsilon_n^2)^2&=\sum_{k=0}^{n-1}\int_{t_k}^{t_{k+1}}\big\|\Bigl[e^{(t_n-t)\IL_h}-\IA_s(\tau\IL_h)^{n-1-k}\IB_s(\tau\IL_h)\Bigr]P_h Q^{\frac12}\big\|_{\mathcal{L}_2(H)}^2dt\\
&\le 3\Bigl((\epsilon_n^{2,1})^2+(\epsilon_n^{2,2})^2+(\epsilon_n^{2,3})^2\Bigr),
\end{align*}
where
\begin{align*}
(\epsilon_n^{2,1})^2&=\int_{0}^{t_{n}}\big\|e^{(t_n-t)\IL_h}(I-\IB_s(\tau\IL_h))P_hQ^{\frac12}\big\|_{\mathcal{L}_2(H)}^2dt\\
(\epsilon_n^{2,2})^2&=\sum_{k=0}^{n-1}\int_{t_k}^{t_{k+1}}\big\|\bigl(e^{(t_n-t)\IL_h}-e^{(t_n-t_{k+1})\IL_h}\bigr)\IB_s(\tau\IL_h)P_hQ^{\frac12}\big\|_{\mathcal{L}_2(H)}^2dt\\
(\epsilon_n^{2,3})^2&=\tau\sum_{k=0}^{n-1}\big\|\bigl(e^{(t_n-t_{k+1})\IL_h}-\IA_s(\tau\IL_h)^{n-1-k}\bigr)\IB_s(\tau\IL_h))P_hQ^{\frac12}\big\|_{\mathcal{L}_2(H)}^2
\end{align*}
We next prove upper bounds for the quantities $(\epsilon_n^{2,j})^2, j=1,2,3$ as follows.
\begin{itemize}
\item Estimate of $\epsilon_n^{2,1}$. Owing to Proposition~\ref{pro:32}, \eqref{eq:regul_expo}, and to~\eqref{eq:pro_integrator:B_error},
\begin{align*}
(\epsilon_n^{2,1})^2&\le\int_{0}^{t_{n}}\big\|e^{(t_n-t)\IL_h}(I-\IB_s(\tau\IL_h))(-\IL_h)^{\frac{1-\alpha-\epsilon}{2}}\big\|_{\mathcal{L}(V_h)}^2 dt \big\|(-\IL_h)^{\frac{\alpha+\epsilon-1}{2}}P_hQ^{\frac12}\big\|_{\mathcal{L}_2(H)}^2\\
&\le C\int_{0}^{t_{n}}\|(-\IL_h)^{\frac{1-\epsilon}{2}}e^{(t_n-t)\IL_h}\|_{\mathcal{L}(V_h)}^2 dt \|(I-\IB(\tau \IL_h))(-\IL_h)^{-\frac{\alpha}{2}}\|_{\mathcal{L}(V_h)}^2\\
&\le C\int_{0}^{t_n}t^{\epsilon-1}dt \tau^\alpha\\
&\le C\tau^\alpha.
\end{align*}
\item Estimate of $\epsilon_n^{2,2}$. Owing to Proposition~\ref{pro:32},
 \eqref{eq:regul_expo_time}, and to~\eqref{eq:pro_integrator:B},
\begin{align*}
(\epsilon_n^{2,2})^2&\le C\sum_{k=0}^{n-2}\int_{t_k}^{t_{k+1}}\|\bigl(e^{(t_{k+1}-t)\IL_h}-I)(-\IL_h)^{-\frac{\alpha}{2}}\|_{\mathcal{L}(V_h)}^2 dt\|e^{(t_n-t_{k+1})\IL_h}(-\IL_h)^{\frac{1-\epsilon}{2}}\big\|_{\mathcal{L}(V_h)}^2\\
&+C\int_{t_{n-1}}^{t_{n}}\|\bigl(e^{(t_{n}-t)\IL_h}-I)(-\IL_h)^{-\frac{\alpha}{2}}\big\|_{\mathcal{L}(V_h)}^2 dt \big\|\IB_s(\tau\IL_h)(-\IL_h)^{\frac{1-\epsilon}{2}}\big\|_{\mathcal{L}(V_h)}^2\\
&\le C\tau^{\alpha} \bigl(\tau\sum_{k=0}^{n-2}\frac{1}{\bigl((n-1-k)\tau\bigr)^{1-\epsilon}}+\tau^\epsilon\bigr)\\
&\le C\tau^\alpha.
\end{align*}
\item Estimate of $\epsilon_n^{2,3}$. Owing to Assumption~\ref{ass:Q}, \eqref{eq:pro_integrator:A_expo}, and to~\eqref{eq:pro_integrator:B},
\begin{align*}
(\epsilon_n^{2,3})^2&\le C\tau\sum_{k=0}^{n-2}\|\bigl(e^{(t_n-t_{k+1})\IL_h}-\IA_s(\tau\IL_h)^{n-1-k}\bigr)\|_{\mathcal{L}(V_h)}^2\|\IB_s(\tau\IL_h)(-\IL_h)^{\frac{1-\epsilon-\alpha}{2}}\|_{\mathcal{L}(V_h)}^2\\
&\le C\tau\sum_{k=0}^{n-2}\frac{1}{(n-1-k)^{1-\epsilon}} \tau^{\alpha+\epsilon-1}\\
&\le C\tau^\alpha\bigl(\tau\sum_{k=0}^{n-2}\frac{1}{\bigl((n-1-k)\tau\bigr)^{1-\epsilon}}\bigr)\\
&\le C\tau^\alpha.
\end{align*}
\end{itemize}
\noindent{\it Proof of~\eqref{eq:claim3}.}
The error term $\epsilon_n^3$ is decomposed as follows:
\[
\epsilon_n^3\le \epsilon_n^{3,1}+\epsilon_n^{3,2}+\epsilon_n^{3,3}+\epsilon_n^{3,4},
\]
where
\begin{align*}
\epsilon_n^{3,1}&=\sum_{k=0}^{n-1}\int_{t_k}^{t_{k+1}}\bigl(\IE\big|e^{(t_n-t)\IL_h}\bigl(F_h(u^h(t))-F_h(u_k^h)\bigr)\big|^2\bigr)^{\frac12} dt\\
\epsilon_n^{3,2}&=\sum_{k=0}^{n-1}\int_{t_k}^{t_{k+1}}\bigl(\IE\big|e^{(t_n-t)\IL_h}\bigl(I-\IB_s(\tau\IL_h)\bigr)F_h(u_k^h)\big|^2\bigr)^{\frac12} dt\\
\epsilon_n^{3,3}&=\sum_{k=0}^{n-1}\int_{t_k}^{t_{k+1}}\bigl(\IE\big|\bigl(e^{(t_n-t)\IL_h}-e^{(t_n-t_{k+1})\IL_h}\bigr)\IB_s(\tau\IL_h)F_h(u_k^h)\big|^2\bigr)^{\frac12} dt\\
\epsilon_n^{3,4}&=\tau\sum_{k=0}^{n-1}\bigl(\IE\big|\bigl(e^{(t_n-t_{k+1})\IL_h}-\IA_s(\tau\IL_h)^{n-1-k}\bigr)\IB_s(\tau\IL_h)F_h(u_k^h)\big|^2\bigr)^{\frac12}.
\end{align*}
We next estimate the quantities $\epsilon_n^{3,1}, j=1,2,3,4$ as follows.
\begin{itemize}
\item Estimate of $\epsilon_n^{3,1}$.
By the Lipschitz continuity of $F_h$ (uniformly with respect to the parameter $h\in(0,1]$), and to the temporal regularity estimate from Proposition~\ref{pro:uh}, one obtains
\begin{align*}
\epsilon_n^{3,1}&\le C\tau\sum_{k=0}^{n-1}\bigl(\IE|u^h(t_k)-u_k^h|^2\bigr)^{\frac12}+C\sum_{k=0}^{n-1}\int_{t_k}^{t_{k+1}}\bigl(\IE|u^h(t)-u^h(t_k)|^2\bigr)^{\frac12}dt\\
&\le C\tau\sum_{k=0}^{n-1}\epsilon_k +C\tau^{\frac{\alpha}{2}}\bigl(\tau^{1-\frac{\alpha}{2}}+\tau\sum_{k=1}^{n-1}(1+\frac{|u_0^h|}{(k\tau)^{\frac{\alpha}{2}}})\bigr).
\end{align*}

\item Estimate of $\epsilon_n^{3,2}$.
Using the linear growth property of $F$, Proposition~\ref{pro:scheme}, and the inequalities~\eqref{eq:pro_integrator:B_error} and~\eqref{eq:regul_expo}, one has
\begin{align*}
\epsilon_n^{3,2}&\le C\|(-\IL_h)^{-\frac{\alpha}{2}}(I-\IB_s(\tau\IL_h))\|_{\mathcal{L}(V_h)}\sum_{k=0}^{n-1}\int_{t_k}^{t_{k+1}}\|e^{(t_n-t)\IL_h}(-\IL_h)^{\frac{\alpha}{2}}\|_{\mathcal{L}(V_h)}dt \bigl(1+\IE|u_k^h|^2\bigr)^{\frac12}\\
&\le C\tau^{\frac{\alpha}{2}}\int_{0}^{t_n}\frac{dt}{(t_n-t)^{\frac{\alpha}{2}}}(1+|u_0|).
\end{align*}

\item Estimate of $\epsilon_n^{3,3}$.
Using the linear growth property of $F$, Proposition~\ref{pro:scheme},  inequality~\eqref{eq:pro_integrator:B}, and using a combination of~\eqref{eq:regul_expo} and~\eqref{eq:regul_expo_time}, then one obtains
\begin{align*}
\epsilon_n^{3,3}&\le C(1+|u_0|)\sum_{k=0}^{n-1}\int_{t_k}^{t_{k+1}}\big\|e^{(t_n-t)\IL_h}-e^{(t_n-t_{k+1})\IL_h}\big\|_{\mathcal{L}(V_h)} dt\\
&\le C(1+|u_0|)\sum_{k=0}^{n-1}\int_{t_k}^{t_{k+1}}\big\|\bigl(e^{(t_{k+1}-t)\IL_h}-I\bigr)e^{(t_n-t_{k+1})\IL_h}\big\|_{\mathcal{L}(V_h)} dt\\
&\le C(1+|u_0|)\tau^{\frac{\alpha}{2}}\bigl(\tau^{1-\frac{\alpha}{2}}+\tau\sum_{k=0}^{n-2}\frac{C}{\bigl((n-1-k)\tau\bigr)^{\frac{\alpha}{2}}}\bigr).
\end{align*}

\item Estimate of $\epsilon_n^{3,4}$.
Using the linear growth property of $F$, and the inequality~\eqref{eq:pro_integrator:B}, and then using~\eqref{eq:pro_integrator:A_expo}, one obtains
\begin{align*}
\epsilon_n^{3,3}&\le C(1+|u_0|)\tau \sum_{k=0}^{n-1}\big\|e^{(t_n-t_{k+1})\IL_h}-\IA_s(\tau\IL_h)^{n-1-k}\big\|_{\mathcal{L}(V_h)}\\
&\le C(1+|u_0|)\tau^{\frac{\alpha}{2}}\bigl(\tau^{1-\frac{\alpha}{2}}+\tau\sum_{k=0}^{n-2}\frac{C}{\bigl((n-1-k)\tau\bigr)^{\frac{\alpha}{2}}}\bigr).
\end{align*}

\end{itemize}
Combining the error estimates~\eqref{eq:claim1},~\eqref{eq:claim2} and~\eqref{eq:claim3}, the error satisfies $\epsilon_0=0$ and
\[
\epsilon_n\le C \tau\sum_{k=0}^{n-1}\epsilon_k+C\tau^{\frac{\alpha}{2}}(1+t_n^{-\frac{\alpha}{2}}|u_0|).
\]
Applying the discrete Gronwall lemma yields the result and this concludes the proof of Theorem~\ref{theo:conv}.
\end{proof}

The proof of Theorem~\ref{theo:convregular} requires a different approach from the proof above of Theorem~\ref{theo:conv} to obtain order $1$ of strong convergence for the time discretization of the SPDE~\eqref{eq:SPDE} driven by additive noise (instead of order at most $1/2$ in Theorem~\ref{theo:conv}).

\begin{proof}[Proof of Theorem \ref{theo:convregular}]
Let us first establish that there exists $C\in(0,\infty)$ such that for all $s\ge 1$, $z\in[-L_s,0]$ and $n\ge 1$, one has
\begin{equation} \label{eq:claim}
|\IA_s(z)^n-e^{nz}|\le C\min(1,n|z|^2)(|\IA(z)|^{n-1}+e^{(n-1)z}).
\end{equation}
On the one hand, $\underset{s\ge 1,z\in[-L_s,0]}\sup~|\IA(z)|<\infty$ by  \eqref{eq:ass_integrator:damping}, thus
\[
|\IA_s(z)^n-e^{nz}|\le |\IA_s(z)|^n+e^{nz}\le C|\IA_s(z)|^{n-1}+e^{(n-1)z}.
\]
On the other hand, using $b^m-a^m\le mb^{m-1}(b-a)$ for all $0\le a\le b$ and all $m\ge 1$, we deduce
\begin{align*}
|\IA_s(z)^n-e^{nz}|&\le n |\IA_s(z)-e^z|(|\IA_s(z)|^{n-1}+e^{(n-1)z})\\
&\le n|z|^2 |r_{1,s}(z)|(|\IA_s(z)|^{n-1}+e^{(n-1)z}),
\end{align*}
where $r_{1,s}(z)=\frac{\IA_s(z)-e^z}{z^2}$. We use the same techniques as in the proof of Lemma~\ref{lem:integrator}. Let $\delta\in(0,1)$. 
Using \eqref{eq:ass_integrator:values} and the maximum principle one has
\[
\underset{s\ge 1,z\in[-\delta,0]}\sup~|r_{1,s}(z)|<\infty.
\]
In addition, using~\eqref{eq:ass_integrator:damping}, one has
\[
\underset{s\ge 1,z\in[-L_s,-\delta]}\sup~|r_{1,s}(z)|\le \frac{1}{\delta^2}(\underset{s\ge 1,z\in [-L_s,-\delta]}\sup~|\IA_s(z)|+e^{-\delta})<\infty.
\]
This concludes the proof of \eqref{eq:claim}.


It remains to prove the error estimate \eqref{eq:convregular}.
Note that, due to Assumption~\ref{ass:FE}, one has the following variant of the result of Proposition~\ref{pro:32},
\[
\underset{h\in(0,1]}\sup~\|(-\IL_h)^{\frac{1}{2}}P_hQ^{\frac12}\big\|_{\mathcal{L}_2(H)}\le \|(-\IL)^{\frac12}Q^{\frac12}\|_{\mathcal{L}_2(H)}<\infty,
\]
where we have used the assumption~\eqref{eq:conditionregular}.

We then have the following decomposition of the error, which is analogous to that in the proof of Theorem \ref{theo:conv}, but with $\epsilon_n^3=0$ due to $F=0$,
\[
\bigl(\IE|u^h(n\tau)-u_n^h|^2\bigr)^{\frac12}\le \epsilon_n^1+\epsilon_n^2,
\]
with
\[
\epsilon_n^2\le 3\bigl((\epsilon_n^{2,1})^2+(\epsilon_n^{2,2})^2+(\epsilon_n^{2,3})^2\bigr).
\]
The term $\epsilon_n^1$ is treated using~\eqref{eq:pro_integrator:A_expo}:
\[
|\epsilon_n^1|\le \frac{C}{n}|u_0^h|\le \frac{C\tau|u_0|}{t_n}.
\]
In addition, using the inequalities from Proposition~\ref{pro:integrator}, the terms $\epsilon_n^{2,1}$ and $\epsilon_n^{2,2}$ can be treated similarly:
\begin{align*}
(\epsilon_n^{2,1})^2&=\int_{0}^{t_{n}}\big\|e^{(t_n-t)\IL_h}(I-\IB_s(\tau\IL_h))P_hQ^{\frac12}\big\|_{\mathcal{L}_2(H)}^2dt\\
&\le \int_{0}^{t_{n}}\big\|e^{(t_n-t)\IL_h}(I-\IB_s(\tau\IL_h))(-\IL_h)^{-\frac{1}{2}}\big\|_{\mathcal{L}(V_h)}^2 dt \big\|(-\IL_h)^{\frac{1}{2}}P_hQ^{\frac12}\big\|_{\mathcal{L}_2(H)}^2\\
&\le C\int_{0}^{t_n}\big\|(-\IL_h)^{\frac{1}{2}-\epsilon}e^{t\IL_h}\big\|_{\mathcal{L}(V_h)}^2dt \big\|(I-\IB_s(\tau\IL_h))(-\IL_h)^{-1+\epsilon} \big\|_{\mathcal{L}(V_h)}^2\\
&\le C(T)\tau^{2(1-\epsilon)},
\end{align*}
and analogously,
\begin{align*}
(\epsilon_n^{2,2})^2&=\sum_{k=0}^{n-1}\int_{t_k}^{t_{k+1}}\big\|\bigl(e^{(t_n-t)\IL_h}-e^{(t_n-t_{k+1})\IL_h}\bigr)\IB_s(\tau\IL_h)P_hQ^{\frac12}\big\|_{\mathcal{L}_2(H)}^2dt\\
&\le C\sum_{k=0}^{n-2}\int_{t_k}^{t_{k+1}}\|\bigl(e^{(t_{k+1}-t)\IL_h}-I)(-\IL_h)^{-1+\epsilon}\|_{\mathcal{L}(V_h)}^2 dt\|e^{(t_n-t_{k+1})\IL_h}(-\IL_h)^{\frac{1}{2}-\epsilon}\big\|_{\mathcal{L}(V_h)}^2\\
&+C\int_{t_{n-1}}^{t_{n}}\|\bigl(e^{(t_{n}-t)\IL_h}-I)(-\IL_h)^{-1+\epsilon}\big\|_{\mathcal{L}(V_h)}^2 dt \big\|\IB_s(\tau\IL_h)(-\IL_h)^{\frac{1}{2}-\epsilon}\big\|_{\mathcal{L}(V_h)}^2\\
&\le C\bigl(\tau\sum_{k=0}^{n-2}\frac{\tau^{2(1-\epsilon)}}{\bigl((n-1-k)\tau\bigr)^{1-\epsilon}}+\tau^2\bigr).
\end{align*}
It remains to deal with $\epsilon_{n}^{2,3}$, using different arguments from the proof of Theorem~\ref{theo:conv}.
\begin{align*}
(\epsilon_n^{2,3})^2&=\tau\sum_{k=0}^{n-1}\big\|\bigl(e^{(t_n-t_{k+1})\IL_h}-\IA_s(\tau\IL_h)^{n-1-k}\bigr)\IB_s(\tau\IL_h))P_hQ^{\frac12}\big\|_{\mathcal{L}_2(H)}^2\\
&\le\tau\sum_{k=0}^{\infty}\sum_{m=1}^{N_h}|Q^{\frac12}e_{m,h}|^2 \big|e^{-k\tau\lambda_{m,h}}-\IA_s(-\tau\lambda_{m,h})^{k}\big|^2|\IB_s(-\tau\lambda_{m,h})|^2\\
&\le C\tau\sum_{m=1}^{N_h}|Q^{\frac12}e_{m,h}|^2\min(1,\tau^4\lambda_{m,h}^4)\sum_{k=0}^{\infty}k^2\bigl(e^{-2k\tau\lambda_{m,h}}+|\IA_s(-\tau\lambda_{m,h})|^{2k}\bigr),
\end{align*}
where we have used the claim~\eqref{eq:claim}.

Since $\underset{x\in[0,1)}\sup~(1-x)^3\sum_{k=0}^{\infty}k^2 x^k<\infty$ and $\min(1,y^4)\le y\min(1,y^3)$ for all $y\ge 0$, one obtains
\begin{align*}
(\epsilon_n^{2,3})^2&\le C\tau \sum_{m=1}^{N_h}|Q^{\frac12}e_{m,h}|^2\min(1,\tau^4\lambda_{m,h}^4)\bigl(\frac{1}{(1-e^{-2\tau\lambda_{m,h}})^3}+\frac{1}{(1-|\IA_s(-\tau\lambda_{m,h})|^2)^3}\bigr)\\
&\le C\tau^2\sum_{m=1}^{N_h}|Q^{\frac12}(-\IL_h)^{\frac12}e_{m,h}|^2\bigl(\frac{\min(1,\tau\lambda_{m,h})^3}{(1-e^{-2\tau\lambda_{m,h}})^3}+\frac{\min(1,\tau\lambda_{m,h})^3}{(1-|\IA_s(-\tau\lambda_{m,h}|^2)^3}\bigr).
\end{align*}
In addition, from assumption~\eqref{eq:newassump}, we deduce
\[
\underset{s\ge 1,z\in[-L_s,0]}\sup~\bigl(\frac{\min(1,|z|)}{1-e^{-2z}}+\frac{\min(1,|z|)}{1-|\IA_s(z)|^2} \bigr)<\infty.
\]
Finally, this yields
$$
(\epsilon_n^{2,3})^2\le C\tau^2\sum_{m=1}^{N_h}|Q^{\frac12}(-\IL_h)^{\frac12}e_{m,h}|^2
\le C\tau^2\|(-\IL)^{\frac12}Q^{\frac12}\|_{\mathcal{L}_2(H)}^2.
$$
Gathering the above estimates concludes the proof of Theorem \ref{theo:convregular}.
\end{proof}

\begin{remark} \label{rem:nonzeroF}
The statement of Theorem \ref{theo:convregular} could be extended for a non zero $F$, assuming that $F$ is of class $\mathcal{C}^2$ from $H$ to itself with bounded first and second derivatives. In the proof of the main Theorem \ref{theo:conv} one would need to change the way $\epsilon_n^{3,1}$ is treated (with this approach one can not overcome the order $1/2$). Performing a second-order Taylor expansion and using the stochastic Fubini theorem would give the result. The details are standard and are omitted for brevity.
\end{remark}

\section{Numerical experiments} \label{sec:numexp}

\begin{figure}[tb]
		\centering
		\begin{subfigure}[t]{0.49\textwidth}
			\centering
			\includegraphics[width=1\linewidth]{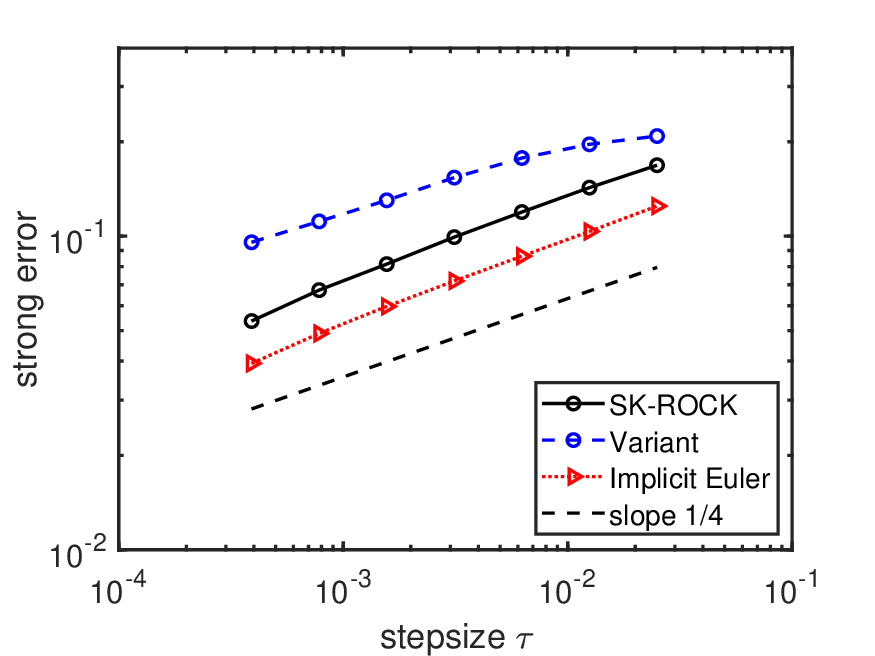}
			\caption{Additive noise (\eqref{eq:heat1d}, $g(u)=1$).}
		\end{subfigure}
		\hfill
		\begin{subfigure}[t]{0.49\textwidth}
			\centering
			\includegraphics[width=1\linewidth]{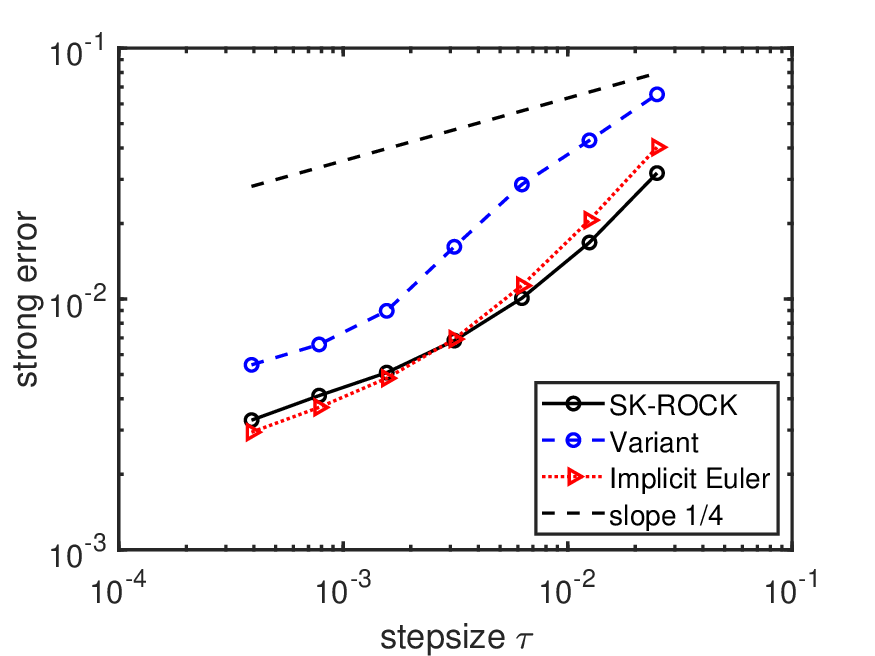}
			\caption{Multiplicative noise (\eqref{eq:heat1d-multip}, $g(u)=u$).}
		\end{subfigure}
		\caption{
		Stochastic heat equation \eqref{eq:heat1d} in dimension $d=1$ with space-time white noise.
		Strong convergence plots using SK-ROCK \eqref{eq:meth1}, the variant \eqref{eq:meth2}, and the implicit Euler method \eqref{eq:meth_eul_impl}, with final time $T=0.1$. Averages over $10^4$ samples.}
		\label{fig:plotconv}
	\end{figure}
In this section, we illustrate our convergence analysis on several test problems based on the semilinear stochastic heat equations in dimensions $d=1,2$.
\subsection{Semilinear stochastic heat equations in dimensions one and two}
We shall first consider
the following stochastic heat equation with additive space-time white noise,
on the domain $\mathcal{D}=(0,1)$  in dimension $d=1$,
\begin{equation}\label{eq:heat1d}
\partial_t u(x,t) = \partial_{xx} u(x,t) + f(u(x,t)) + \partial_t W(x,t),
\end{equation}
and the stochastic heat equation on the domain $\mathcal{D}=(0,1)^2$ in dimension $d=2$, where we consider a space-time white-noise only in the direction $x_1$ of space, where $x=(x_1,x_2)\in \mathcal{D}$,
\begin{equation}\label{eq:heat2d}
\partial_t u(x,t) = \partial_{x_1x_1} u(x,t) + \partial_{x_2x_2} u(x,t) + f(u(x,t)) + \partial_t W(x_1,t).
\end{equation}
For simplicity, we consider in both cases homogeneous Dirichlet boundary conditions,
$u(x,t)=0$ for all $x\in\partial \mathcal{D}$, and we use the initial condition 
$u(x,0)=\prod_{i=1}^d\sin(2\pi x_i)$. 
We note that both problems \eqref{eq:heat1d} and \eqref{eq:heat2d} satisfy the considered analytical setting of Section~\ref{sec:setting}, where $\overline \alpha=1/2$ in Assumption~\ref{ass:Q}.
We also recall that in dimension $d=2$ or higher, one cannot consider in the stochastic heat equation \eqref{eq:heat2d} a space-time white noise in all space directions (otherwise $\overline \alpha =0$).
Although this is not covered by our analysis, we shall also consider the stochastic heat equation with multiplicative space-time white noise in dimension $d=1$,
\begin{equation}\label{eq:heat1d-multip}
\partial_t u(x,t) = \partial_{xx} u(x,t) + f(u(x,t)) + u(x,t) \partial_t W(x,t).
\end{equation}

For the spatial discretization, we use a standard finite difference method\footnote{Note that such a standard finite difference method on a uniform mesh can be seen as the simplest finite element method which is thus covered by our analysis.} 
with constant mesh size $h=1/100$ for discretizing the Laplacian and the white noise, and obtain the following system where $u(x_i,t)\simeq u_i(t)$ with $x_i=ih, i=1,\ldots N$, $h=1/N$,
$$
du_i=\frac{u_{i+1}-2u_i+u_{i-1}}{h^2} dt + f(u_i) dt + \frac1{\sqrt h} g(u_i) dw_i,\ i=1,\ldots,N-1,
$$
where $w_i,i=1,\ldots N$ are independent Wiener processes, and $u_0=u_N=0$ to take into account the homogeneous Dirichlet boundary conditions, with $g(u)=1$ for~\eqref{eq:heat1d} and $g(u)=u$ for~\eqref{eq:heat1d-multip}.
Analogously, we consider for the spatial discretization of problem \eqref{eq:heat2d}, where $u(x_i,x_j,t)\simeq u_{ij}(t)$, the usual discrete Laplacian with five points for approximating the Laplace operator,
$$
du_{ij}=\frac{u_{i+1,j}+u_{i,j+1}-4u_{ij}+u_{i-1,j}+u_{i,j-1}}{h^2} dt + f(u_{ij}) dt + \frac1{\sqrt h} dw_i,\ i,j=1,\ldots,N-1.
$$
For the SK-ROCK method \eqref{eq:meth1} and the variant \eqref{eq:meth2}, we consider the usual damping parameter $\eta=0.05$ and for achieving the stability condition \eqref{eq:CFL}, the number $s$ of internal stages is computed adaptively (see e.g. \cite{AAV18}) as\footnote{The notation $[\,]$ stands for the integer rounding of real numbers.}
$$s=\left[\sqrt{\frac{\tau \rho_h +1.5}{2-\frac43 \eta}}+0.5\right],$$
where $\rho_h$ is an (upper) estimate for the spectral radius of
the discrete diffusion operator (here $\rho_h = 4dh^{-2}$ for the discrete Laplacian on the domain $\mathcal{D}=(0,1)^d$ in dimensions~$d=1,2$).

\subsection{Convergence comparison and qualitative behavior}
In Figure \ref{fig:plotconv}, we plot the convergence curves for the strong error $\bigl(\IE|u^h(n\tau)-u_n^h|^2\bigr)^{1/2}$ both for the additive noise case~\eqref{eq:heat1d} (see Fig.\ts \ref{fig:plotconv}(a)) and the multiplicative noise case~\eqref{eq:heat1d-multip} (see Fig.\ts \ref{fig:plotconv}(b)) using the nonlinearity $f(u)=-u-\sin(u)$. 
We used averages over $10^4$ trajectories and a reference solution was computed with the small timestep $\tau = 0.1\cdot 2^{-14}$. 
We observe in Figure \ref{fig:plotconv} lines of slope 1/4, this corroborates the strong convergence estimate of Theorem~\ref{theo:conv} in the additive noise case and suggests that the strong convergence estimate persists in the multiplicative noise case (recall that $\overline{\alpha}=1/2$ in this case). We see that the error constant for the SK-ROCK method \eqref{eq:meth1} is better than that of the variant \eqref{eq:meth2}. We also observe that the linear implicit Euler method has a slightly better error constant in the additive noise case, but not in the multiplicative noise case. Note however that this experiment is meant to illustrate the theory, rather than the performance of explicit stabilized methods which typically reveal efficient in larger spatial dimensions, see e.g. \cite{AbV13}.
\begin{figure}[tb]
		\centering
		\begin{subfigure}[t]{0.49\textwidth}
			\centering
			\includegraphics[width=1\linewidth]{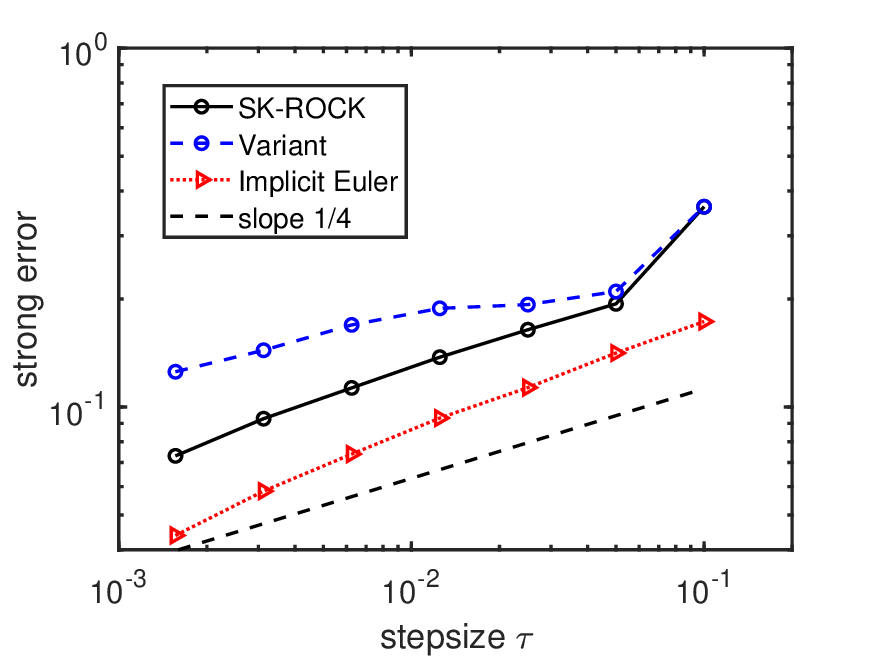}
			\caption{Error versus time stepsize.}
		\end{subfigure}
		\hfill
		\begin{subfigure}[t]{0.49\textwidth}
			\centering
			\includegraphics[width=1\linewidth]{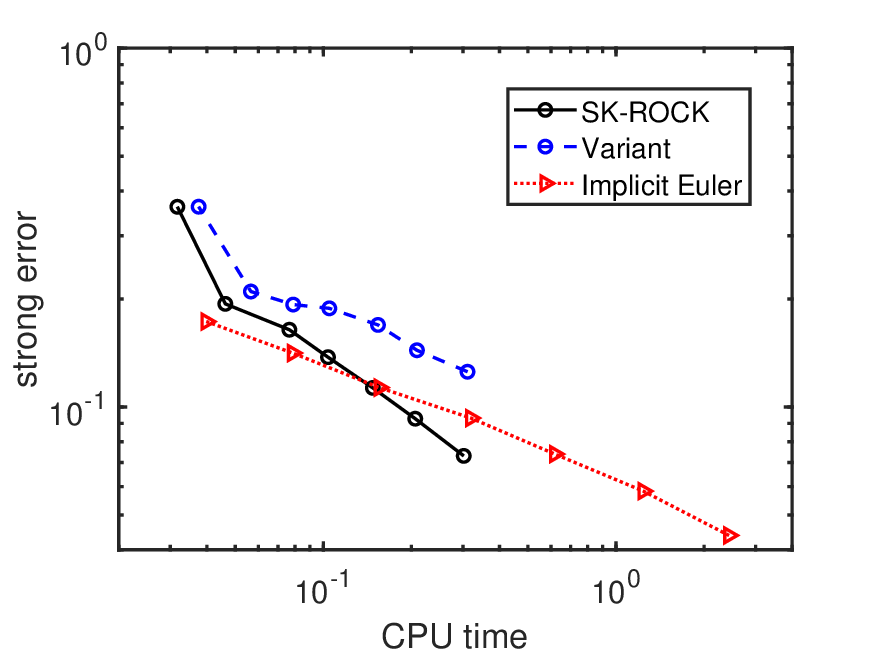}
			\caption{Error versus CPU time.}
		\end{subfigure}
		\caption{
		Stochastic heat equation \eqref{eq:heat2d} in dimension $d=2$. Strong convergence plots using SK-ROCK \eqref{eq:meth1}, the variant \eqref{eq:meth2}, and the implicit Euler method \eqref{eq:meth_eul_impl}. Final time $T=0.1$. Spatial mesh size $h=1/N=10^{-2}$. Averages over $10^2$ samples.}
		\label{fig:plotconv2d}
	\end{figure}
	
In Figure \ref{fig:plotconv2d}, we then consider the two dimensional case \eqref{eq:heat2d} with additive noise, and we plot with the same spacial mesh parameter $h=1/N=1/100$ (corresponding here to a spacial mesh of $N^d=10^4$ points) the strong error $\bigl(\IE|u^h(n\tau)-u_n^h|^2\bigr)^{1/2}$ versus the stepsize~$\tau$ (see Figure \ref{fig:plotconv2d}(a)), and versus the CPU time computed in a (non-parallel) Matlab implementation (see Figure \ref{fig:plotconv2d}(b)). 
Recall that the considered number $s$ of internal stages for the explicit stabilized methods varies with the stepsize $\tau=T/2^{i},i=0,\ldots 6$, and it is given respectively by 
$s=65,46,33,24,17,12,9$.
While Figure \ref{fig:plotconv2d}(a) corroborates the strong
strong convergence estimate of Theorem~\ref{theo:conv},
it can be seen in Figure \ref{fig:plotconv2d}(b) that the explicit stabilized method SK-ROCK indeed reveals competitive compared to the implicit Euler method \eqref{eq:meth_eul_impl} when the CPU time is taken into accounts.

\begin{figure}[tb]
\small
\centering
\begin{subfigure}[t]{0.49\textwidth}
			\centering
			\includegraphics[width=1.\textwidth]{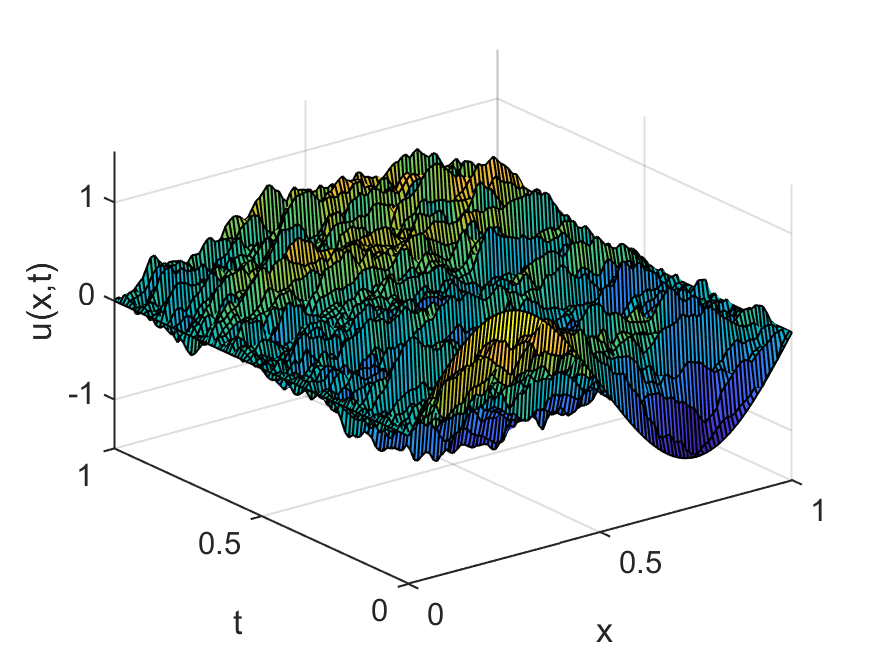}
			\caption{SK-ROCK method \eqref{eq:meth1}.}
		\end{subfigure}
		\hfill
		\begin{subfigure}[t]{0.49\textwidth}
			\centering
			\includegraphics[width=1.\textwidth]{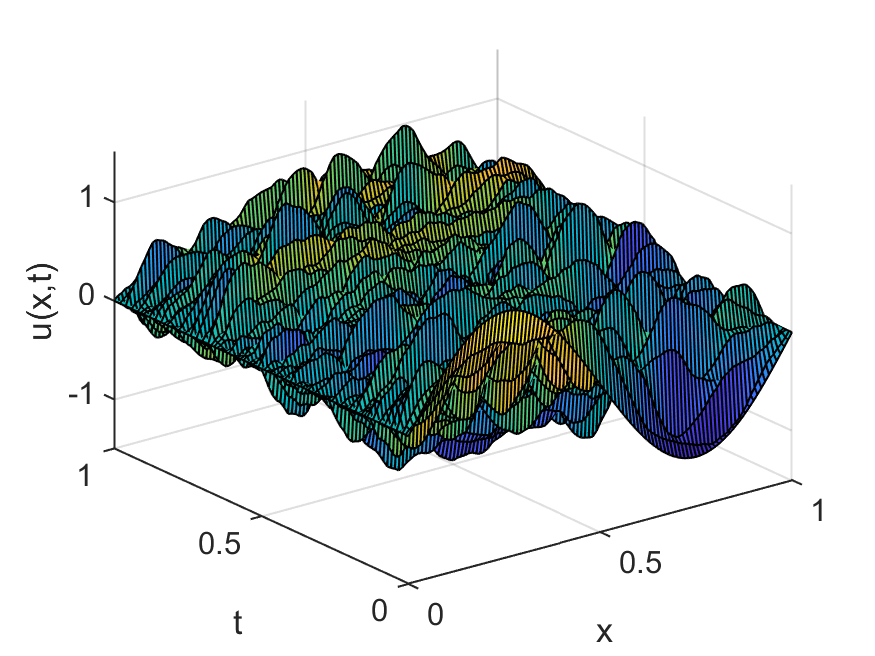}
			\caption{Variant method \eqref{eq:meth2}.}
		\end{subfigure}
		\begin{subfigure}[t]{0.49\textwidth}
			\centering
			\includegraphics[width=1.\textwidth]{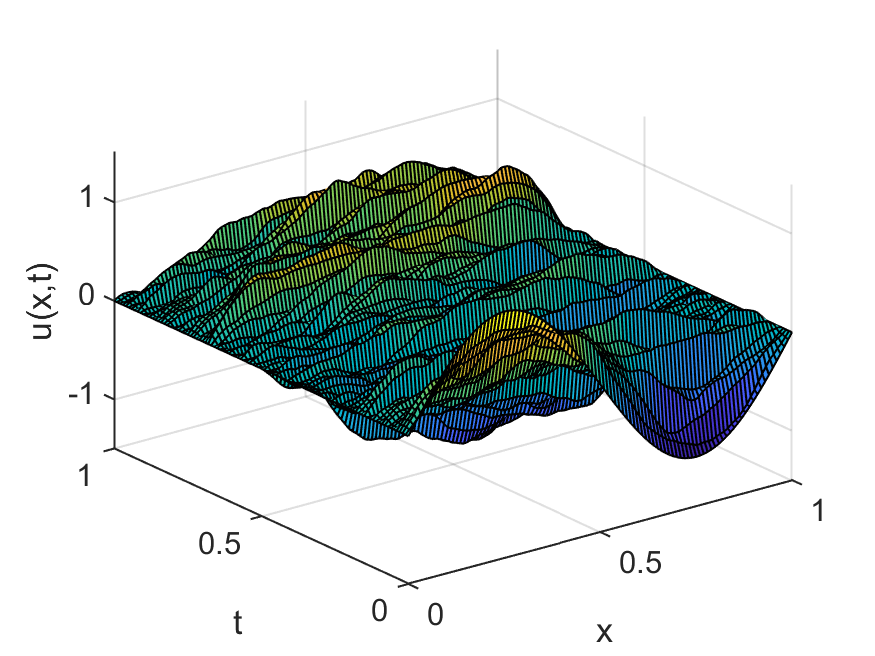}
			\caption{Linear implicit Euler method \eqref{eq:meth_eul_impl}.}
		\end{subfigure}
		\hfill
				\begin{subfigure}[t]{0.49\textwidth}
			\centering
\includegraphics[width=1.\textwidth]{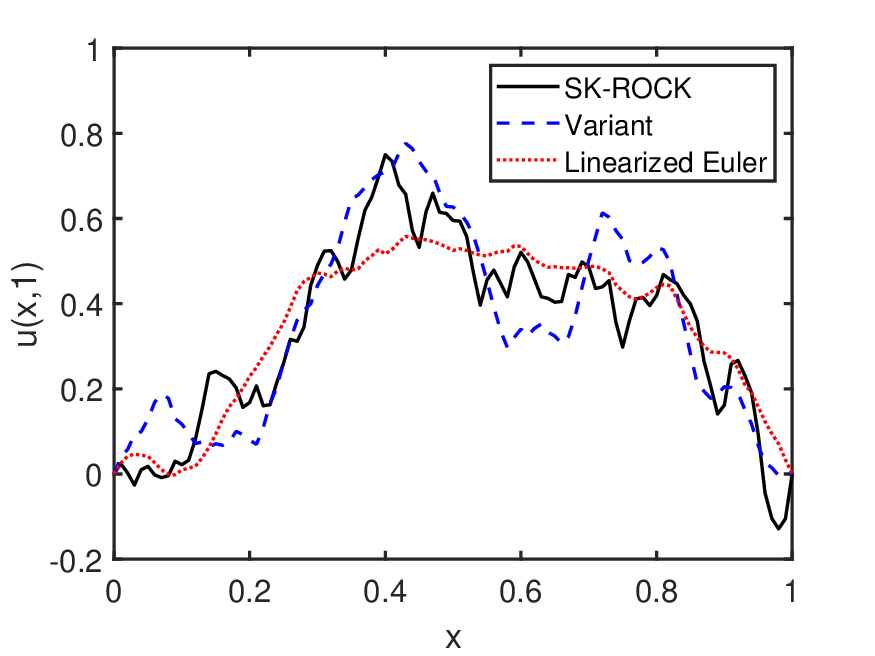}
			\caption{Solution profile at final time $T=1$.}
		\end{subfigure}
\caption{
Samples of realisation of the stochastic nonlinear heat equation \eqref{eq:heat1d} with nonlinearity $f(u)=-u-\sin(u)$ and additive noise $g(u)=1$ using explicit stabilized methods \eqref{eq:meth1}, \eqref{eq:meth2}. With time and spatial stepsizes $\tau=h=1/100$.
\label{fig:figtraj}}
\end{figure}

In Figure \ref{fig:figtraj}, we plot one realization of the stochastic nonlinear heat equation \eqref{eq:heat1d} with $f(u)=-u-\sin(u)$ using the two explicit stabilized methods \eqref{eq:meth1}, \eqref{eq:meth2} and the linear implicit Euler method \eqref{eq:meth_eul_impl}, respectively. For comparison, we used the same random realizations for sampling the noise. We also plot in Figure \ref{fig:figtraj}(d)the corresponding profiles at final time $T=1$. It can be observed that compared to the SK-ROCK method \eqref{eq:meth1} the variant method  \eqref{eq:meth2} and the 
linear implicit Euler method \eqref{eq:meth_eul_impl}
exhibit an increased regularity. Note that increasing the damping parameter $\eta$ of the explicit stabilized methods would increase the regularity of the numerical solutions (this is not illustrated here for brevity).
We mention that the question of the spatial regularity of the numerical solution is addressed in \cite[Prop.\ts 3.9]{BrV16}, where it is shown that the same regularity as the exact solution can be recovered by introducing a suitable postprocessor for the linear implicit Euler method applied to the stochastic heat equation.

\bigskip
\noindent
{\bf Acknowledgements.} 
The work of AA was partially supported by the Swiss National Science Foundation,
project No. 200020\_172710.
The work of C.-E.~B. was partially supported by the project SIMALIN (ANR-19-CE40-0016) operated by the French National Research Agency.
The work of GV was partially supported by the Swiss National Science Foundation,
projects No. 200020\_184614, No. 200020\_192129 and No. 200020\_178752.
The computations were performed at the University of Geneva on the Baobab cluster.

\appendix

\section{Appendix}

In this appendix, we first provide spatial and temporal regularity properties on the process $\bigl(u(t)\bigr)_{t\ge 0}$ and its semi-discrete approximation $\bigl(u^h(t)\bigr)_{t\ge 0}$. The proofs are omitted since they are quite standard and can be found for instance in~\cite{DaPrato_Zabczyk:14} or~\cite{Kruse:14}. 
Let us first recall the following well-posedness and time regularity results, 
see for instance~\cite[Chap.\ts 2, Thm. 2.31]{Kruse:14B} and \cite[Chap.\ts10, Thm.\ts 10.26, Thm.\ts 10.27]{Lord_Powell_Shardlow:14}, and~\cite{DaPrato_Zabczyk:14}.
\begin{pro}\label{pro:wellposed}
Let Assumptions~\ref{ass:F} and~\ref{ass:Q} be satisfied.

For any initial condition $u_0\in H$, there exists a unique global mild solution $\bigl(u(t)\bigr)_{t\ge 0}$ of~\eqref{eq:SPDE}: the process $u$ is continuous with values in $H$ and satisfies~\eqref{eq:mild-SPDE} for all $t\ge 0$.
Moreover, assume that $|(-\IL)^{\frac{\alpha_0}{2}}u_0|<\infty$, with $\alpha_0\in[0,\overline{\alpha})$ (recall that $\overline{\alpha}$ is defined in Assumption~\ref{ass:Q}). For every $\alpha\in[\alpha_0,\overline{\alpha})$ and $T\in(0,\infty)$, there exists a constant $C\in(0,\infty)$ such that for all $t\in(0,T]$,
\[
\IE|(-\IL)^{\frac{\alpha}{2}}u(t)|^2\le C\bigl(1+\frac{|(-\IL)^{\frac{\alpha_0}{2}}u_0|^2}{t^{\alpha-\alpha_0}}\bigr),
\]
and for all $0<t_1\le t_2\le T$,
\[
\IE|u(t_2)-u(t_1)|^2\le C|t_2-t_1|^{\alpha}\bigl(1+\frac{|(-\IL)^{\frac{\alpha_0}{2}}u_0|^2}{t_1^{\alpha-\alpha_0}}\bigr).
\]
\end{pro}

The well-posedness part of the result is obtained by a standard fixed point argument and the computation below with $\alpha=0$. For the spatial regularity estimate, note that the stochastic contribution is treated as follows: by the It\^o isometry formula,
\begin{align*}
\IE|(-\IL)^{\frac{\alpha}{2}}\int_{0}^{t}e^{(t-s)\IL}dW^Q(s)|^2&=\int_{0}^{t}\big\|(-\IL)^{\frac{\alpha}{2}}e^{(t-s)\IL}Q^{\frac12}\big\|_{\mathcal{L}_2(H)}^2ds\\
&\le \int_{0}^{t}\big\|(-\IL)^{\frac{1-\epsilon}{2}}e^{(t-s)\IL}\big\|_{\mathcal{L}(H)}^2ds\|(-\IL)^{\frac{\alpha+\epsilon-1}{2}}Q^{\frac12}\|_{\mathcal{L}_2(H)}^2\\
&\le C\int_{0}^{t}(t-s)^{\epsilon-1}ds \|(-\IL)^{\frac{\alpha+\epsilon-1}{2}}Q^{\frac12}\|_{\mathcal{L}_2(H)}^2, 
\end{align*}
where $\epsilon\in(0,\overline{\alpha}-\alpha)$. In addition, the smoothing properties~\eqref{eq:smoothing} of the semi-group yields
\[
\big|(-\IL)^{\frac{\alpha}{2}}e^{t\IL}u_0\big|\le \|(-\IL)^{\frac{\alpha-\alpha_0}{2}}e^{t\IL}\|_{\mathcal{L}(H)}|(-\IL)^{\frac{\alpha_0}{2}}u_0|\le 
Ct^{-\frac{\alpha_0-\alpha}{2}}|(-\IL)^{\frac{\alpha_0}{2}}u_0|.
\]
For the temporal regularity estimate, similar arguments are used to prove that, first,
\[
\big|\int_{0}^{t_2}e^{(t_2-s)\IL}dW^Q(s)-\int_{0}^{t_1}e^{(t_1-s)\IL}dW^Q(s)\big|^2\le C|t_2-t_1|^\alpha,
\]
and, second, using the smoothing properties~\eqref{eq:smoothing},
\[
|\bigl(e^{t_2\IL}-e^{t_1\IL}\bigr)u_0|\le \|\bigl(e^{(t_2-t_1)\IL}-I\bigr)e^{t_1\IL}(-\IL)^{-\frac{\alpha_0}{2}}\|_{\mathcal{L}(H)}|(-\IL)^{\frac{\alpha_0}{2}}u_0|\le C\frac{|t_2-t_1|^{\frac{\alpha}{2}}}{t_1^{\frac{\alpha-\alpha_0}{2}}}.
\]

We now state without proof the following estimate concerning the regularity properties of $u^h$ defined by~\eqref{eq:SPDE_fem}, and the spatial discretization error. We refer for instance to~\cite[Chap.\ts 2, Thm. 2.27]{Kruse:14B}, \cite[Chap.\ts10, Thm.\ts 10.28]{Lord_Powell_Shardlow:14}, and~\cite{Kruse:14}. 
\begin{pro}\label{pro:uh}
Let Assumptions~\ref{ass:F},~\ref{ass:Q} and~\ref{ass:FE} be satisfied. Assume  $|(-\IL)^{\alpha_0}u_0|<\infty$, with $\alpha_0\in[0,\overline{\alpha})$.
Then,
for every $\alpha\in[\alpha_0,\overline{\alpha})$ and $T\in(0,\infty)$, there exists $C\in(0,\infty)$ such that, for all $t\in(0,T]$,
\[
\underset{h\in(0,1]}\sup~\IE|(-\IL)_h^{\frac{\alpha}{2}}u^h(t)|^2\le C\bigl(1+\frac{|(-\IL)^{\frac{\alpha_0}{2}}u_0|^2}{t^{\alpha-\alpha_0}}\bigr),
\]
and for all $0<t_1\le t_2\le T$,
\[
\underset{h\in(0,1]}\sup~\IE|u^h(t_2)-u^h(t_1)|^2\le C|t_2-t_1|^{\alpha}\bigl(1+\frac{|(-\IL)^{\frac{\alpha_0}{2}}u_0|^2}{t_1^{\alpha-\alpha_0}}\bigr).
\]
For all $\alpha\in [\alpha_0,\overline{\alpha})$ and $T\in(0,\infty)$, there exists $C\in(0,\infty)$, such that for all $t\in(0,T]$,
\[
\bigl(\IE\big|u^h(t)-u(t)\big|^2\bigr)^{\frac12}\le Ch^{\alpha}\bigl(1+\frac{|(-\IL)^{\frac{\alpha_0}{2}}u_0|}{t^{\frac{\alpha-\alpha_0}{2}}}\bigr).
\]
\end{pro}
Proposition \ref{pro:uh} states that the strong order of convergence in space is $\overline{\alpha}$, which is twice the order $\overline{\alpha}/2$ of convergence in time (Theorem \ref{theo:conv}). This is also consistent with the spatial regularity for the process $u$ stated in Proposition~\ref{pro:wellposed}.



%



\end{document}